\documentclass[11pt]{amsart}\usepackage[hmargin=40mm,vmargin=40mm]{geometry}
\usepackage[all]{xy}
\usepackage{a4wide}
\usepackage{mathtools}
\usepackage{amssymb}
\usepackage{amsthm}
\usepackage{amsmath}
\usepackage{amscd}
\usepackage{verbatim}
\usepackage[all]{xy}
\usepackage{hyperref}

\usepackage{enumitem}

\usepackage{xcolor}
\usepackage[textsize=scriptsize,colorinlistoftodos]{todonotes}

\usepackage{tikz-cd}
\usepackage{cleveref}
\usepackage{mathrsfs}
\usepackage{lscape}
\usepackage{geometry}

\numberwithin{equation}{section}

\newcommand{\CC}{\mathbb{C}}

\newcommand{\ra}{\rightarrow}
\newcommand{\ZZ}{\mathbb{Z}}
\newcommand{\RR}{\mathbb{R}}
\newcommand{\QQ}{\mathbb{Q}}

\newcommand{\HH}{\mathbb{H}}

\newcommand{\E}{\mathcal{E}}

\newcommand{\J}{\mathcal{J}}

\renewcommand{\P}{\mathbb{P}}
\newcommand{\bbM}{\mathbb{M}}
\newcommand{\abcd}{\left(\begin{smallmatrix} a & b \\ c & d \end{smallmatrix}\right)}

\renewcommand{\vec}[1]{\boldsymbol{#1}}

\numberwithin{equation}{section}
	\newtheorem{Satz}{Satz}[section]
	\newtheorem{Thm}[Satz]{Theorem}
	\newtheorem{Lem}[Satz]{Lemma}
	\newtheorem{Prop}[Satz]{Proposition}

	\theoremstyle{definition} 
	\newtheorem{Def}[Satz]{Definition}
	
	\newtheorem{rmk}[Satz]{Remark}
\date{\today}

\DeclareMathOperator{\GL}{GL}

\DeclareMathOperator{\End}{End}

\DeclareMathOperator{\SL}{SL}

\DeclareMathOperator{\MP}{Mp}
\DeclareMathOperator{\Imm}{Im}

\usepackage[OT2,T1]{fontenc}
\DeclareSymbolFont{cyrletters}{OT2}{wncyr}{m}{n}
\DeclareMathSymbol{\Sha}{\mathalpha}{cyrletters}{"58}

\DeclareMathOperator{\Sym}{Sym}

\begin{document}

\author{%
Claudia Alfes-Neumann
}
\thanks{The first author is partially supported by the Daimler and Benz Foundation, the Klaus Tschira Boost Fund and funded by the Deutsche Forschungsgemeinschaft (DFG, German Research Foundation) -- SFB-TRR 358/1 2023 -- 491392403}
\author{
Jens Funke%
}
\author{
Michael. H.\ Mertens%
}
\author{
Eugenia Rosu
}
\thanks{The fourth author was partially supported by the Deutsche Forschungs-Gemeinschaft (DFG) through the Collaborative Research Centre TRR 326 "Geometry and Arithmetic of Uniformized Structures", project number 444845124.
}

\address{Universität Bielefeld, Fakultät für Mathematik, Postfach 100 131, 33501 Bielefeld, Germany,
E-Mail: \url{alfes@math.uni-bielefeld.de}}

\address{Durham University, Department of Mathematical Sciences, Science Laboratories, South Rd, Durham DH1 3LE, United Kingdom, E-Mail: \url{jens.funke@durham.ac.uk }}

\address{Universit\"at zu K\"oln, Department Mathematik/Informatik, Abteilung Mathematik, Weyertal 86--90, 50931 K\"oln, Germany, E-Mail: \url{mmertens@math.uni-koeln.de}}

\address{Leiden University, Mathematical Institute, Niels Bohrweg 1, 2333 CA Leiden Leiden, The Netherlands
E-Mail: \url{e.rosu@math.leidenuniv.nl}}

\title{On Jacobi--Weierstrass mock modular forms}

\begin{abstract}
We construct harmonic weak Maass forms that map to cusp forms of weight $k\geq 2$ with rational coefficients under the $\xi$-operator. This generalizes work of the first author, Griffin, Ono, and Rolen, who constructed distinguished preimages under this differential operator of weight $2$ newforms associated to rational elliptic curves using the classical Weierstrass theory of elliptic functions. We extend this theory and construct a vector-valued Jacobi--Weierstrass $\zeta$-function which is a generalization of the classical Weierstrass $\zeta$-function.
\end{abstract}
\maketitle

\section{Introduction}

Harmonic weak Maass forms are real-analytic generalizations of classical modular forms with applications in combinatorics, number theory, as well as representation theory and physics (see for instance \cite{kenbook} and the references therein). They were first introduced by Bruinier and the second author \cite{bruinierfunke} in the context of theta lifts. A \textit{harmonic weak Maass form of weight $k\in\frac{1}{2}\ZZ$} for a congruence subgroup $\Gamma$ of $\SL_2(\ZZ)$ is a smooth function $F:\HH\to\CC$ that transforms like a classical modular form under $\Gamma$, but which is harmonic rather than holomorphic (see Section~\ref{secHMF} for a precise definition). If $F$ has poles in $\HH$ it is called \textit{polar harmonic weak Maass form}.

Harmonic weak Maass forms are inherently connected to deep number-theoretic questions. A  beautiful example is Zwegers' work \cite{zwegers} who showed that Ramanujan's mock theta functions are holomorphic parts of harmonic weak Maass forms. Such functions are also called mock modular forms \cite{zagierbourbaki}. Harmonic weak Maass forms also appear prominently in the Kudla program and numerous other places in mathematics.

The $\xi$-operator
 \[
\xi_k=-2i v^k \overline{\frac{\partial}{\partial\overline{\tau}}}, \ \ \tau=u+iv,
\] 
 plays a crucial role in relating the theory of classical modular forms to harmonic weak Maass forms. 
It defines a surjective map from the space $H_{k}$ of harmonic weak Maass forms of weight $k$ to the space $S_{2-k}$ of cusp forms of weight $2-k$.
The image under the $\xi$-operator is called the \textit{shadow} of the corresponding harmonic weak Maass form resp.\ mock modular form. 
In particular, there are infinitely many preimages in the space $H_k$ for any given cusp form in $S_{2-k}$, yet they are surprisingly difficult to construct. The quest for distinguished preimages under the $\xi$-operator is a fundamental question, and as cusp forms are omnipresent in number theory, the construction of distinguished preimages offers new routes to tackle related problems.

\smallskip
There are various approaches to this problem: Bruinier \cite{brhabil} and Bringmann and Ono \cite{bringmannono} showed that certain real-analytic Poincar\'e series, originally introduced by Niebur \cite{niebur}, map to the classical exponential type Poincar\'e series of dual weight under the $\xi$-operator, provided that they converge, e.g.\ when the weight is negative.  The coefficients of cuspidal Poincar\'e series are however not easy to handle, both practically and theoretically.

Another approach to the problem uses the \emph{non-holomorphic Eichler integral} 
$$g^*(\tau)=\int_{-\overline \tau}^\infty \frac{\overline{f(-\overline z)}}{(-i(z-\tau))^{k}}dz$$ 
for a cusp form $g$ of weight $2-k$, as proposed and used first in work by Andrews, Rhoades, and Zwegers \cite{ARZ13} as well as Dabholkar, Murthy, and Zagier \cite{DBZ}. Using a suitable auxiliary holomorphic cusp form $h$ and holomorphic projection, they construct a harmonic weak Maass form $F(\tau)$ of weight $k$ such that $\xi_{k}F=cg$ for some constant $c$. 
This approach is particularly well-suited when $g$ is a unary theta function, since then the Fourier coefficients of the holomorphic projection can be evaluated explicitly (see also  \cite{MOR21,MMR22}).

Recently, Ehlen, Li, and Schwagenscheidt \cite{ELS22} gave an explicit procedure to find a \emph{good} preimage of a CM modular form using a certain theta lift. It is known from previous work by Bruinier, Ono, and Rhoades \cite{BOR08} that such good preimages always exist and in particular that their holomorphic parts have algebraic Fourier coefficients at $\infty$.  Using their explicit construction, Ehlen, Li, and Schwagenscheidt could pinpoint the exact algebraic number field containing these Fourier coefficients.

\smallskip

In this work we focus on the geometric approach and generalize the work of Griffin, Ono, Rolen, and the first author \cite{agor}. They construct harmonic weak Maass forms whose shadows are given by newforms of weight $2$ with rational coefficients, generalizing previous work by Guerzhoy \cite{guerzhoy}.

Let $E$ be an elliptic curve defined over $\QQ$ and let $G_E$ be the corresponding newform of weight $2$ for a congruence subgroup $\Gamma_0(N)$ of $\SL_2(\ZZ)$ via the modularity theorem. Over $\CC$, $E$ is isomorphic to  a torus $\CC/\Lambda_E$, with the isomorphism given  by
\[
\wp'(z)^2=4\wp(z)^3-g_2\wp(z)+g_3,
\]
where $\wp$ is the Weierstrass $\wp$-function for the lattice $\Lambda_{E}$, and $g_2, g_3$ are the normalized Eisenstein series for the lattice $\Lambda_E$. 

The \emph{Weierstrass $\zeta$-function} 
\begin{equation}
\zeta_{\Lambda_{E}}(z)
=
\frac{1}{z}+\sum_{\omega\in \Lambda_{E}\setminus\{0\}}\left(\frac{1}{z-\omega}+\frac{1}{\omega}+\frac{z}{\omega^2}\right),
\end{equation}
is not itself invariant under shifts by lattice points, but has a well-known non-analytic completion $\zeta_{\Lambda_E}^*(z)$ which is indeed invariant under $\Lambda_E$. This completion allows one to construct a canonical preimage $Z_E$ of the modular form $G_E$ under the operator $\xi_0$ by taking
\[
Z_E(\tau)=\zeta_{\Lambda_E}^*(\E_{E}(\tau)),
\]
where $\E_{E}(\tau)=\int_{\tau}^{\infty}G_E(t) dt$ is the holomorphic Eichler integral. The \textit{error to modularity} of the Eichler integral is determined by the lattice, i.e.\ $\mathcal{E}_E(\gamma\tau)=\mathcal{E}_E(\tau)+\omega$, for $\gamma\in \Gamma_0(N)$ and for some $\omega\in\Lambda_E$, giving the $\Gamma_0(N)$-invariance of $Z_E(\tau)$. We recall the details of this construction in Section~\ref{secWeierstrass}. 

These preimages $Z_E$ can be computed very efficiently and are also of theoretical interest. In \cite{agor}, they were used to obtain a criterion for the vanishing of critical $L$-derivatives of quadratic twists of $E$. Another application in the context of vertex operator algebras can be found in \cite{beneishmertens}.

\smallskip

 In the current work we generalize the above construction to newforms $f$ with rational Fourier coefficients of weight $k\geq 2$, for a congruence subgroup $\Gamma\subset \SL_2(\ZZ)$, with $\Gamma_1(N)\subset \Gamma$ for some integer $N$. Our construction involves vector-valued forms. For simplicity we restrict to the easiest case in the introduction and focus on describing the central ideas of our work. 

We start by noting that a modular form $f$  of weight $k\geq 2$ gives a vector-valued modular form $f:\HH \ra \Sym^{k-2}(\CC^2)$ of weight $2$ by taking
\begin{equation}
G(\tau)=cf(\tau)(\tau e_1+e_2)^{k-2},
\end{equation}
where $(e_1, e_2)$ is a fixed basis of $\CC^2$ and $\Sym^{k-2}(\CC^2)$ is the $(k-2)$-th symmetric power of $\CC^2$, which can be realized as the space of homogeneous polynomials in $e_1$ and $e_2$ of degree $k-2$, and $c$ is a constant. Note that this construction of vector-valued modular forms goes back to classic work of Kuga and Shimura \cite{kush}.  For further details, see Section~\ref{sec:reptheory}.

The construction involves two central elements: a vector-valued generalization of the Weierstrass $\zeta$-function $\widehat{\zeta}$ and the polynomial Eichler integral $\E_f$.

 In the simplest case (see Section~\ref{sec:laurent}), we define the multivariable function $\zeta: \HH\times \CC^{k-1} \ra \Sym^{k-2}(\CC^2)$ by 
 \begin{equation}
\zeta(\tau, \vec{z})=\sum_{j=0}^{k-2} \zeta_{\Lambda_{\tau}}(z_j) c_j e_1^j e_2^{k-2-j}, \ \ c_j={k-2 \choose j}, \ \ \Lambda_{\tau}=\ZZ+\ZZ\tau.
\end{equation}
It has a natural completion
\begin{equation}
\widehat{\zeta}(\tau, \vec{z})=\sum_{j=0}^{k-2} \zeta^*_{\Lambda_{\tau}}(z_j) c_j e_1^j e_2^{k-2-j},
\end{equation}
where $\zeta_{\Lambda_{\tau}}^*$ is the completion of the classical Weierstrass $\zeta$-function.  Our first main result is the general construction of {\it Jacobi--Weierstrass $\zeta$-functions} $\widehat{\zeta}$.

\begin{Thm} The function $\widehat{\zeta}$ is a (non-holomorphic) Jacobi form of weight $1$ and index $0$, invariant under the lattice $\Lambda_{\tau}^{k-1}$. 
\end{Thm}

To construct the higher degree generalization $\widehat{\zeta}$ of the completed Weierstrass $\zeta$-function we follow Rolen \cite{Rolen15}. The key point is to write the $\zeta$-function as the logarithmic derivative of the Weierstrass $\sigma$-function which in turn is essentially given by a multiple of Jacobi's theta function. This definition of the $\sigma$-function can be generalized to higher degree by giving an analogous construction using a Jacobi theta function of lattice index. 
Replacing the logarithmic derivative by a suitable weight raising operator we obtain the Jacobi--Weierstrass $\zeta$-function.

\smallskip

We then use the Jacobi--Weierstrass $\zeta$-function to construct the natural preimages of $f$ by evaluating it at the polynomial Eichler integral $\E_f: \mathbb{H}\times \SL_2(\ZZ)\to \Sym^{k-2}(\CC^2)$,
\begin{equation}
\E_f(\tau, (X_1,X_2))
=
\int_{\tau}^{\infty} f(t) (tX_1+X_2)^{k-2} dt, 
\end{equation}
for $\tau\in\HH$, $M=\left(\begin{smallmatrix} a& b \\ c& d \end{smallmatrix} \right)\in \SL_2(\ZZ)$,  $(X_1, X_2)=(e_1 ,e_2) M$ is the change of basis of $\CC^2$, and $f$ is a cusp form of weight $k$.

To realize this, we first extend the Jacobi--Weierstrass $\zeta$-function to include a possible change of basis for the space $\CC^2$, given by $M=\left(\begin{smallmatrix} a& b \\ c& d \end{smallmatrix} \right)\in \SL_2(\ZZ)$, as follows
\begin{equation}
\widehat{\zeta}_{M}(\tau, \vec{z})=\sum_{j=0}^{k-2} \zeta^*_{\Lambda_{\tau}}(z_j) X_1^j X_2^{k-2-j},  \ \ (X_1, X_2)=(e_1, e_2) M.
\end{equation}
This gives the vector-valued Jacobi form $\widehat{\zeta}: \HH \times \CC^{k-1} \times \SL_2(\ZZ) \ra \Sym^{k-2}(\CC^2)$ where we write $\widehat{\zeta}(\tau, \vec{z}, M)=\widehat{\zeta}_{M}(\tau, \vec{z}).$
Fixing further the standard basis of $\CC^2$, we take the natural isomorphism $\Sym^{k-2}(\CC^2)\simeq \CC^{k-1}$, which gives us the most general Jacobi--Weierstrass function $\widehat{\zeta}: \HH \times\Sym^{k-2}(\CC^2)\times \SL_2(\ZZ)  \ra \Sym^{k-2}(\CC^2).$ This form was carefully defined to account for the various actions of the subgroup $\Gamma$ on the upper-half plane $\mathbb{H}$ and $\Sym^{k-2}(\CC^2)$ (compare Section~\ref{vector_valued}).

Finally, we define the vector-valued form $F:\HH \times \SL_2(\ZZ) \ra \Sym^{k-2}(\CC^2)$ by taking
\begin{equation}
F(\tau, M)=\widehat{\zeta}_{M}(\Lambda_f, \E(\tau, (e_1 ,e_2)M^{-1})),
\end{equation}
where $\Lambda_f$ is the lattice corresponding to $f$ (see Section~\ref{sec:eichlerint}, in particular Remark~\ref{remlambdaf}).  

We note that $F$ can also be written as a vector-valued form
$
F:\HH \ra V, 
$
valued in the infinite-dimensional vector space $V=\mathrm{Functions}\left\{\SL_2(\ZZ)\to\Sym^{k-2}(\CC^2)\right\}$, by sending $F(\tau)$ to the function $M \ra F(\tau, M)$ in $V$.

The main result of the paper is the following:

\begin{Thm}\label{main_intro}
The function $F:\HH \ra V$ is a vector-valued polar harmonic weak Maass form of weight $0$ for $\Gamma$, whose image under $\xi_0$ is given by
\[
G(\tau)=\frac{2\pi i}{\mathrm{Vol}(\Lambda_f)}f(\tau)(\tau e_1+e_2)^{k-2}.
\]
\end{Thm}

Note that in the case of $k=2$ the functions we obtain through our construction are in fact constant in the $\SL_2(\ZZ)$-variable, therefore we recover the results from \cite{agor}. The poles of $F$ are explicitly computed in Proposition \ref{poles}.

\begin{rmk}
The naive approach of fixing a basis $(e_1, e_2)$ of $\CC^2$ and plugging in the Eichler integral directly $\widehat{\zeta}(\Lambda_f, \E_f(\tau, (e_1, e_2)))$ loses the $\Gamma$-invariance. In general we do not have a compatibility between the action of $\Sym^{k-2}(\CC^2)$ on the outside and on the inside
\[
\widehat{\zeta}\left(\Lambda_f, \E_f\left(\tau, \left(e_1, e_2\right)\right)\right)
\neq 
\gamma\cdot \widehat{\zeta}\left(\Lambda_f, \E_f\left(\gamma \tau, \left(e_1, e_2\right)\right)\right).
\]
By controlling the change of basis with the extra variable $M\in \SL_2(\ZZ)$ we insure the $\Gamma$-invariance of the vector-valued form $F$.
\end{rmk}

\smallskip

The paper is organized as follows. In Section \ref{sec:prelim} we introduce the notion of harmonic weak Maass forms, give the necessary background on representation theory, and introduce the vector-valued Eichler integral. In Section \ref{secWeierstrass} we review the construction of Weierstrass harmonic weak Maass forms in the case of weight $2$. The generalization to higher degree of the completed Weierstrass $\zeta$-function is given in Section \ref{sec:vvJW}. Section \ref{vector_valued} contains our main result, namely the construction of the vector-valued Jacobi--Weierstrass harmonic weak Maass form $F:\HH\ra V$. In Section \ref{sec:laurent} we compute the Laurent expansion of the Jacobi--Weierstrass $\zeta$-function and the poles of the polar harmonic weak Maass form $F$. In Section \ref{examples} we provide two examples of our construction.

\section{Preliminaries}\label{sec:prelim}

Throughout this paper we let $\Gamma$ be a congruence subgroup of $\SL_2(\ZZ)$ such that $\Gamma_1(N)\subset \Gamma$ for some positive integer $N$, and $\Gamma_1(N)=\{\gamma\in  \SL_2(\ZZ): \gamma \equiv \left(\begin{smallmatrix} 1 & * \\ 0& 1\end{smallmatrix}\right) \mod N\}$. We call $f$ a newform for $\Gamma$ if it is a newform for $\Gamma_1(N)$.

\subsection{Harmonic weak Maass forms}\label{secHMF}
By $\mathrm{Mp}_2(\RR)$ we denote the metaplectic group consisting of pairs $(\gamma,\phi)$, where $\gamma=\abcd\in\SL_2(\RR)$ and $\phi:\HH\to\CC$ is a holomorphic function with $\phi(\tau)^2=c\tau+d$. We let $\mathrm{Mp}_2(\ZZ)$ be the inverse image of $\SL_2(\ZZ)$ under the covering map $\mathrm{Mp}_2(\RR)\to \SL_2(\RR)$.

A twice continuously differentiable function $F:\HH\to\CC$ is called a \textit{harmonic weak Maass form of weight $k\in\frac12\ZZ$} for $\mathrm{Mp}_2(\ZZ)$ if it satisfies
\begin{enumerate}
\item $\Delta_k F=0$, where $\Delta_k=-v^2\left(\frac{\partial^2}{\partial u^2}+\frac{\partial^2}{\partial v^2}\right)+ikv\left(\frac{\partial}{\partial u}+i\frac{\partial}{\partial v}\right)$ is the hyperbolic weight $k$ Laplace operator and we write $\tau=u+iv\in\HH$.
\item $F(\tau)|_k(\gamma,\phi):=\phi(\tau)^{-2k} F(\gamma\tau)= F(\tau)$ for $(\gamma,\phi)\in\mathrm{Mp}_2(\ZZ)$.
\item There is a Fourier polynomial $P_F(\tau)=\sum_{n\leq 0} a^+(n)q^n$, called the principal part of $F$, such that
\[
F(\tau)-P_F(\tau)=O(e^{-\varepsilon v})
\]
as $v\to\infty$, uniformly in $u$, for some $\varepsilon>0$. A similar conditions holds at all cusps.
\end{enumerate}
We define harmonic weak Maass forms of integral weight $k\in\ZZ$ transforming with respect to the group $\Gamma$ accordingly. If $F$ has poles on $\HH$ it is called \textit{polar harmonic weak Maass form}.

\subsection{The Jacobi theta function}\label{sec:jacobi}
Let $L\simeq \ZZ^g$ be an even lattice with positive definite inner product defined through $\left(v,w\right):=v^{t}G_Lw$, where $G_L$ is the Gram matrix of $L$, a symmetric matrix with integer entries and even entries on the diagonal,  and let $V=L\otimes \CC\simeq \CC^g$ with quadratic form $Q(v)=\left(v, v\right)/2$. 
	\begin{Def} 
	Let $\tau\in\mathbb{H}$ and $\vec{z}\in\CC^g$. We define the Jacobi theta function by
	\[
	\theta(\tau, \vec{z})=\sum_{\ell\in L} e^{\pi i(\ell, \ell)\tau} e^{2\pi i (\ell, \vec{z})}.
	\]
	\end{Def}	
The following two transformation properties of $\theta(\tau,\vec z)$ are well-known
	\begin{align}\label{lattice}
	\theta(\tau, \vec{z}+\vec{m}\tau+\vec{n})&=e^{-\pi i \tau (
\vec{m}, \vec{m})}e^{-2\pi i (\vec{z}, \vec{m})}\theta(\tau, \vec{z}), \ \ \vec m,\vec n\in \ZZ^g,
\\
	\theta\left(\tau+1, \vec{z}\right)
	&=
	\theta(\tau, \vec{z}).
	\end{align}
	 Under the additional assumption that $L$ is unimodular, see e.g.\ \cite{boylan}, it follows from Poisson summation that
	\begin{align}
	\theta\left(-\frac 1\tau,\frac{\vec z}\tau\right)&=
    (\tau/i)^{g/2}e^{2\pi i\frac{Q(\vec{z})}{\tau}}\theta(\tau,\vec{z}).
	\end{align}
	
These functional equations imply that $\theta(\tau,\vec z)$ is a Jacobi form of weight $g/2$ and index $\frac 12 G_L$ in the sense of \cite{boechererkohnen} for the full Jacobi group $\SL_2(\ZZ)\ltimes (\ZZ^g)^2$.  If $L$ is not unimodular, $\theta(\tau,\vec z)$ is a Jacobi form for some suitable congruence subgroup $\mathcal{J}_{L}=\Gamma_{L}\ltimes (\ZZ^g)^2$ of the full Jacobi group (compare \cite[Corollary 3.34, 3.35]{boylan}).

We follow B\"ocherer--Kohnen \cite{boechererkohnen} and define the \textit{slash operator} by
\begin{align*}
(\phi|_{k,mL}A) (\tau,\vec{z})& :=\phi\left( \frac{a\tau+b}{c\tau +d},\frac{\vec{z}+\vec{m}\tau+\vec{n}}{c\tau+d}\right) (\sqrt{c\tau+d})^{-2k}  \\
&\quad \quad\quad  e^{2\pi i m \left(-c\frac{ 2Q(\vec{z}+\vec{m}\tau+\vec{n})}{(c\tau +d)}+\left(\vec{m},\vec{m}\right)\tau +2\left(\vec{m},\vec{z}\right)\right)},
\end{align*}
where $A=\left[\left(\left(\begin{smallmatrix} a& b \\ c& d \end{smallmatrix}\right),\sqrt{c\tau+d}\right),(\vec{m},\vec{n})\right] \in \MP_2(\ZZ)\ltimes(\ZZ^g)^2$ and $k,m\in\frac{1}{2}\ZZ$.
Then the condition that $\theta$ is a Jacobi form for the Jacobi subgroup $\J_L$ is equivalent to
\begin{equation}\label{jacobi_theta}
\theta|_{g/2, 1/2} A=\theta,  \ \ \text{ for }A\in\mathcal{J}_L.
\end{equation}
	
As in \cite{bringmannraumrichter} we now define a weight raising operator for Jacobi forms by 
\begin{equation}
Y_{+, z_j}^{g/2, 1/2L}=\frac{\partial}{\partial z_j}+2\pi i \frac{\Imm((G_L \vec{z})_j)}{\Imm(\tau)},
\end{equation}
where $\vec{z}=(z_1,\ldots,z_g)\in\CC^g$. 
As we will show in Proposition~\ref{propzetaell}, we have
\begin{equation}\label{jacobi}
Y_{+, z_j}^{g/2, 1/2L}\theta
=
(Y_{+, z_j}^{g/2, 1/2L}\theta)|_{g/2+1, 1/2} A
\end{equation}
for $A=\left[\left(\left(\begin{smallmatrix} a& b \\ c& d \end{smallmatrix}\right),\sqrt{c\tau+d}\right),(\vec{m},\vec{n})\right] \in \J_L$ and $k \in\frac{1}{2}\ZZ$.

\subsection{A little bit of representation theory}\label{sec:reptheory}
We let $W(=\CC^2)$ be the standard complex representation of $\SL_2(\ZZ)$ and let $\langle\, , \,\rangle $ be the usual symplectic form on $W$ with standard basis $e_1,e_2$. This defines an integral structure on $W$. We let $W_m=\Sym^m W$ be the irreducible representation of dimension $m+1$ of highest weight $m$. Weight vectors are (multiples of) $e_1^je_2^{m-j}$ with corresponding weight $j$ with respect to the standard rationally split torus $S=\{\left(\begin{smallmatrix} t&0\\0&t^{-1}\end{smallmatrix}\right)\,:\,t\in\RR\}$. Here we write $v^k$ for the symmetric tensor $v\bullet v\bullet \cdots \bullet v$.

Note that $W_m$ is self-dual, i.e.\ $W_m\simeq W_m^*$, where the isomorphism is induced by the symplectic form on $W$. We will not distinguish between $W_m$ and $W_m^*$.

The action of $\SL_2(\ZZ)$ on $W_{k-2}=\Sym^{k-2}(\CC^2)$ is given by acting on the basis on $\CC^2$, i.e.\
\[
M \circ (e_1, e_2):= 
\begin{pmatrix}
e_1 &e_2
\end{pmatrix}
M= \begin{pmatrix}
a e_1 +ce_2 & be_1 +de_2
\end{pmatrix},
 \, \text{ for }M =\left( \begin{smallmatrix}
a&b\\c&d
\end{smallmatrix}\right)\in \SL_2(\ZZ).
\]

We recall the usual action of $\SL_2(\ZZ)$ of the upper-half plane $\HH$ by fractional linear transformations $M  \tau=\frac{a\tau+b}{c\tau+d}$ for $M=\left(\begin{smallmatrix} a& b \\ c& d\end{smallmatrix}\right)\in\SL_2(\ZZ)$. Moreover, we denote by $j(M, \tau)=c\tau+d$ the factor of automorphy.

We then have
\begin{equation}\label{eq:actiononsymandtau}
M^{-1}\circ((M \tau) e_1+e_2)^{k-2}
=
j(M, \tau)^{2-k}(\tau e_1+e_2)^{k-2}
\end{equation}
for $M\in \SL_2(\ZZ)$.

For a modular form $f$ of weight $k$ for the congruence subgroup $\Gamma$, this directly implies that the holomorphic $1$-form on $\HH$ given by
\[
\eta_f:= f(\tau)d\tau\otimes (\tau e_1+e_2)^{k-2}
\]
satisfies $\gamma^{-1} \circ \eta_f(\gamma \tau)=\eta_f(\tau)$ for $\gamma\in\Gamma$.

\subsection{Eichler integrals}\label{sec:eichlerint}

Throughout this section let $f$ be a cusp form of weight $k$ for some finite index subgroup $\Gamma$ of $\SL_2(\ZZ)$ and let $(X_1, X_2)$ be a basis for $W$. We define the \emph{Eichler integral of $f$} by
\begin{equation}
\E_f(\tau, (X_1, X_2))=\int_{\tau}^{\infty} f(t) (tX_1+X_2)^{k-2} dt.
\end{equation}
It defines a function $\E_f:\HH\times \SL_2(\ZZ) \ra  \Sym^{k-2}(\CC^2)$, where the basis $(X_1, X_2)$ above corresponds to a matrix $M\in \SL_2(\ZZ)$ such that $(X_1,X_2)=M\circ (e_1 ,e_2)$. Note that we recover the usual (scalar-valued) Eichler integral by identifying $W_m$ with the space of homogeneous polynomials of degree $m$ in two variables $X,Y$ and by replacing $X_1$ by $1$ and $X_2$ by $\tau$.

For $\gamma \in \Gamma$ we consider
\[
 \E_f(\gamma\tau, \gamma^{-1}\circ (X_1, X_2)) = \int_{\gamma\tau}^{\infty} f(t) \left[\gamma^{-1}\circ(tX_1+X_2)^{k-2} \right]dt,  \,\gamma\in\Gamma.
\]

\begin{Lem} \label{action_gamma}
Let $f$ be a cusp form of weight $k$ for $\Gamma$, we have 
\[
 \E_f(\gamma\tau, \gamma^{-1}\circ (X_1, X_2))
=
\E_f(\tau, (X_1, X_2))- \sum_{\ell=0}^{k-2} {k-2 \choose \ell} X_1^\ell X_2^{k-2-\ell} \int_{\gamma^{-1}\infty}^{\infty} f(t) t^\ell dt.
\]
for $\gamma\in \Gamma$ and $(X_1, X_2)$ a basis of $\CC^2$.
\end{Lem}
\begin{proof}
Identity \eqref{eq:actiononsymandtau} implies
\begin{align*}
 \E_f(\gamma \tau, \gamma^{-1}\circ (X_1, X_2))&=\int_{\gamma\tau}^{\infty} f(t) \cdot\left[ \gamma^{-1}\circ(tX_1+X_2)^{k-2} \right]dt
\\
&=
\int_{\tau}^{\gamma^{-1} \infty} f(t)(t X_1+X_2)^{k-2} dt.
\end{align*}
We split the integral 
\[
\int_{\tau}^{\gamma^{-1} \infty} f(t)  (t X_1+X_2)^{k-2}  dt
=
\int_{\tau}^{\infty} f(t)  (t X_1+X_2)^{k-2}  dt
-
\int_{\gamma^{-1} \infty}^{\infty} f(t)  (t X_1+X_2)^{k-2}  dt
\]
and use the binomial expansion to obtain the result.
\end{proof}

We now recall some facts about modular symbols. Further details and background may be found for instance in Chapters $8$ and $10$ in \cite{stein}. Let $\bbM_{k}$ denote the module of \emph{modular symbols}, i.e.\ the $\ZZ$-module generated by symbols $P\{\alpha,\beta\}$, where $P\in\ZZ[X,Y]_{k-2}$ is a homogeneous polynomial of degree $k-2$ with integer coefficients and $\alpha,\beta\in\P^1(\QQ)$ are cusps, modulo the relations obtained from 
$$\{\alpha,\beta\}+\{\beta,\gamma\}+\{\gamma,\alpha\}=0$$
and all torsion.
The group $\Gamma$ acts from the left on modular symbols by combining the usual action on $\P^1(\QQ)$ via M\"obius transformations and on homogeneous polynomials defined by
$$(\abcd,P)\mapsto P(dX-bY,-cX+aY).$$
We denote by $\bbM_k(\Gamma)$ the module of modular symbols modulo all torsion and all relations obtained from $x-\gamma.x$ where $x\in\bbM_k$ and $\gamma\in\Gamma$. Every cusp form $f\in S_k(\Gamma)$ defines a pairing
\begin{gather}\label{eqPhi}
\Phi_f:\bbM_k(\Gamma)\to \CC, \ \ (f,P\{\alpha,\beta\})\mapsto \int_\alpha^\beta f(t)P(t,1)dt,
\end{gather}
which is consistent with the action of Hecke operators. If $f$ is a normalized newform with integral coefficients, then the image $\Phi_f(\bbM_k(\Gamma))$ is a lattice $\Lambda_f$ in $\CC$.

\begin{rmk}\label{remlambdaf}
We write 
\begin{equation}
\E_f(\tau, (X_1, X_2))
=
\sum_{\ell=0}^{k-2} \E_{\ell, f}(\tau) c_\ell X_1^\ell X_2^{k-2-\ell},
\end{equation}
with $c_\ell={k-2 \choose \ell}$ and $\E_{\ell, f}(\tau)=\int_{\tau}^{\infty}f(t) t^\ell dt$.  If $f$ has integral coefficients, the considerations above imply that 
\begin{equation}
\E_{\ell, f}(\gamma\tau) - \E_{\ell, f}(\tau) \in \Lambda_f.
\end{equation}
For $k=2$ and $\Gamma=\Gamma_0(N)$, the lattice $\Lambda_f$ coincides with the period lattice of the rational elliptic curve associated to $f$.
\end{rmk}

We now choose the basis $\{c_\ell X_1^\ell X_2^{k-2-\ell}\}_{0\leq \ell\leq k-2}$ of $\Sym^{k-2}(\CC^2)$. In these coordinates the Eichler integral is given as
\begin{equation}
\E_f(\tau, (X_1, X_2))=(\E_{0, f}(\tau), \dots, \E_{k-2, f}(\tau))_{(X_1, X_2)}.
\end{equation}
For later convenience, we reformulate Lemma~\ref{action_gamma} in a slightly different fashion.
\begin{Lem}\label{gamma_inv} 
The Eichler integral $\E_f(\tau, (X_1, X_2))$ is $\Lambda_f^{k-1}$-invariant under the $\Gamma$-action, meaning that for $\gamma\in\Gamma$ we have 
\[
(\E_{0, f}(\gamma\tau), \dots, \E_{k-2, f}(\gamma\tau))_{\gamma^{-1}\circ(X_1, X_2)}
=
(\E_{0, f}(\tau), \dots, \E_{k-2, f}(\tau))_{(X_1, X_2)}+\omega,
\]
with $\omega=(\lambda_0, \dots, \lambda_{k-1})_{(X_1, X_2)}$ and each $\lambda_i\in \Lambda_f$.
\end{Lem}

A straightforward computation gives the Fourier expansion of the coefficients $\mathcal E_{f,\ell}(\tau)$.
\begin{Lem}
For a cusp form $f\in S_{k}(\Gamma)$ with a Fourier expansion $f(\tau)=\sum\limits_{n\geq 1}a_nq^n$ we have 
$$\mathcal E_{f,\ell}(\tau)=\int_\tau^\infty f(t)t^\ell dt=\sum_{n\geq 1} \frac{a_n}{2\pi n}\left(\sum_{j=0}^\ell \frac{\ell!}{(\ell-j)!}\left(\frac{i}{2\pi n}\right)^j\tau^{\ell-j} \right)q^n.$$
\end{Lem}

As is also apparent directly from the definition, $\mathcal E_{f,\ell}$ is not 1-periodic except when $\ell=0$. In general the Fourier coefficients of $\E_{f, \ell}$ are polynomials in $\tau$.
\subsection{Change of basis}\label{sec:changeofbasis}
We note that each basis $(X_1, X_2)$ of $\CC^2$ gives rise to a basis

\noindent $\{c_\ell X_1^{\ell}X_2^{k-2-\ell}\}_{0\leq \ell\leq k-2}$ of $\Sym^{k-2}(\CC^2)$. Let $\vec{z}=(z_0,\ldots_,z_{k-2})\in\CC^{k-1}$. We write
\begin{equation}
(\vec{z})^{\mathrm{t}}_{(X_1, X_2)}=(z_0,\ldots,z_{k-1})^{\mathrm{t}}_{(X_1,X_2)}=\sum\limits_{\ell=0}^{k-2} z_\ell c_\ell X_1^{\ell}X_2^{k-2-\ell}
\end{equation}
for the corresponding element in $\Sym^{k-2}(\CC^2)$. 
Then the action of $\SL_2(\ZZ)$ is given by
\[
M \circ (\vec{z})^{\mathrm{t}}_{(e_1, e_2)}
	:=
	(\vec{z})^{\mathrm{t}}_{M \circ  (e_1, e_2)}.
	\]
This gives a change of basis from the standard basis $\{c_\ell e_1^{\ell}e_2^{k-2-\ell}\}_{0\leq \ell\leq k-2}$ of $\Sym^{k-2}(\CC^2)$ to the basis $\{c_\ell X_1^{\ell}X_2^{k-2-\ell}\}_{0\leq \ell\leq k-2}$, where $(X_1, X_2)=M \circ(e_1, e_2)$.  We describe this change of basis explicitly.
\begin{Lem}\label{n(M)}
Let $M=\abcd\in\SL_2(\ZZ)$. Then
\[
(\vec{z})^{\mathrm{t}}_{M  \circ (e_1, e_2)} =N(M) (\vec{z})^{\mathrm{t}}_{  (e_1, e_2)}
\]
for a matrix $N(M) \in  \GL_{k-1}(\ZZ)$, whose entries $(N_{\ell t})_{0\leq \ell, t \leq k-2}$ are given by
\[
N_{\ell t}=\sum\limits_{\substack{0\leq i\leq t\\ 0\leq \ell-i\leq k-2-t}} {\ell \choose i}  {k-2-\ell\choose t-i}a^i c^{t-i} b^{\ell-i} d^{k-2-t-\ell+i} .
\]
\end{Lem}

\begin{proof}
We note that
\[
(\vec{z})^{\mathrm{t}}_{M\circ (e_1, e_2)}
=\sum\limits_{t=0}^{k-2} z_t c_t(ae_1+ce_2)^t (be_1+de_2)^{k-2-t}.
\]
By using the binomial theorem twice and rearranging the sums we obtain the stated formula.

The claim that $N(M)\in \GL_{k-1}(\ZZ)$ follows easily from the properties mentioned in Remark~\ref{remNgamma} below: for each $M\in\SL_2(\ZZ)$, $N(M)$ has all integer entries and it is a group homomorphism, which implies that $N(M)^{-1}=N(M^{-1})$ has integer entries as well, proving the claim.
\end{proof}

We note that this formula is already contained in classical work by Kuga-Shimura from the late 1950's \cite{kush}.

\begin{rmk}\label{remNgamma} For $M, M_1, M_2\in\SL_2(\ZZ)$ and the notation as in Lemma \ref{n(M)}, we note the following properties
\begin{align*}
N(M_1M_2)&=N(M_2)N(M_1),\\
 N(M)^\mathrm{t} &= N(M^\mathrm{t}),\\
N(M)^{-1} &=N(M^{-1}).
\end{align*}
\end{rmk}

Since $N(M)\in\GL_{k-1}(\ZZ)$ by Lemma~\ref{n(M)} we immediately obtain the following observation.
\begin{Lem}\label{inv_lambda} 
Let $M \in \SL_2(\ZZ)$ and the notation be as in Lemma \ref{n(M)}. Then we have
\[
N(M) \Lambda_f^{k-1} =\Lambda_f^{k-1}.
\]
Here, we view an element of $\Lambda_f^{k-1}$ as a column vector.
\end{Lem}


\section{Weierstrass mock modular forms}\label{secWeierstrass}
We recall the classical construction of Weierstrass mock modular forms of Guerzhoy \cite{guerzhoy} and Alfes--Griffin--Ono--Rolen \cite{agor}. 

\subsection{The completed Weierstrass $\zeta$-function}\label{seczetanew}
In this section we review Rolen's \cite{Rolen15} construction of Eisenstein's completion of the Weierstrass $\zeta$-function.

We let $\Lambda_\tau =\ZZ +\ZZ\tau$ for $\tau\in\mathbb{H}$ and define for $z\in\CC$ the classical {\it Weierstrass $\sigma$-function} 
\[
\sigma_{\Lambda_\tau}(z) = z\,\prod_{w\in\Lambda_\tau\setminus\{0\}} \left(1-\frac{z}{w}\right)\exp\left(\frac{z}{w}+\frac{z^2}{2w^2}\right).
\]
Its logarithmic derivative is the {\it Weierstrass $\zeta$-function}
\[
\zeta_{\Lambda_\tau}(z)
=
\frac{\frac{\partial}{\partial z} \sigma_{\Lambda_\tau}(z)}{ \sigma_{\Lambda_\tau}(z)}.
\]
Moreover, we need the following classical identity (cf.\ Theorem 3.9 of \cite{poli}) relating the Weierstrass $\sigma$-function to the standard Jacobi theta function 
\begin{equation}\label{eq:ellipticthetasigma}
\vartheta(\tau,z) = -2\pi \eta(\tau)^3 \exp\left(-\frac{\eta_1 z^2}{2}\right) \sigma_{\Lambda_\tau}(z).
\end{equation}
Here, $\vartheta(\tau,z)$ is a Jacobi form of weight $1/2$ and index $1/2$ given by
\[
\vartheta(\tau,z)=\sum_{n\in\frac12+\ZZ} e^{\pi i n^2\tau+2\pi i n \left(z+\frac12\right)},
\]
$\eta_1$ is the quasi-period defined by
\[
\eta_1=\eta_1(\tau):= \zeta_{\Lambda_\tau}(z+1)-\zeta_{\Lambda_\tau}(z),
\]
and $\eta(\tau)=q^{1/24}\prod_{n=1}^\infty (1-q^n)$ is the Dedekind $\eta$-function.

Then we find that
\[
\frac{\frac{\partial}{\partial z}\vartheta(\tau,z)}{\vartheta(\tau,z)} = \zeta_{\Lambda_\tau}(z)-\eta_1 z.
\]
As $\vartheta(\tau,z)$ is a holomorphic function in $z$, one may consider its derivative. However, this is no longer a Jacobi form.
We therefore replace the derivative $\frac{\partial}{\partial z}$ by the canonical raising operator $Y_+^{\frac12,\frac12}$ on Jacobi forms and apply it to the Jacobi theta function
\[
Y_+^{\frac12,\frac12} (\vartheta(\tau,z))=\frac{\partial}{\partial z} \vartheta(\tau,z) +2\pi i \frac{\Imm(z)}{\Imm(\tau)}\vartheta(\tau,z)
\]
which then is a Jacobi form of weight $3/2$ and index $1/2$ (compare Section \ref{sec:jacobi}). Considering the corresponding analogue of the logarithmic derivative we obtain
\begin{equation}\label{eq:logtheta}
\frac{Y_+^{\frac12,\frac12} (\vartheta(\tau,z))}{\vartheta(\tau,z)}= \zeta_{\Lambda_\tau}(z) -\eta_1 z +\frac{2\pi i \Imm(z)}{\Imm(\tau)},
\end{equation}
which is a real-analytic Jacobi form of weight $1$ and index $0$. In particular, it is an elliptic function in $z$. Using the relation $G_2=\eta_1$ for the weight $2$ Eisenstein series, normalized to have constant term $2\zeta(2)$, one can show that the function
\[
\zeta^*_{\Lambda_\tau} (z) 
=
\zeta_{\Lambda_\tau} (z)-zG^*_2(\tau) -\pi \frac{\overline{z}}{\Imm(\tau)}
\]
is a doubly-periodic function with respect to the lattice $\Lambda_\tau$. Here, $G_2^*(\tau)=G_2(\tau)-\pi/\Imm(\tau)$ is the non-holomorphic completion of the weight $2$ Eisenstein series. 

We note that 
\[
 \Imm(\tau)=\mathrm{Vol}(\Lambda_\tau).
\]

\subsection{Weierstrass mock modular forms}
The completed Weierstrass $\zeta$-function can be used to produce harmonic weak Maass forms of weight $0$. We let $f$ be a newform of weight $2$ for $\Gamma_0(N)$ with rational Fourier coefficients. In the notation of Section \ref{sec:eichlerint}, $\Lambda_f$ is the associated lattice and $\mathcal{E}_f(\tau) =\sum_{n=1}^\infty \frac{a_f(n)}{n}q^n$ is the Eichler integral of $ f(\tau)=\sum_{n=1}^\infty a_f(n)q^n$. In particular, we have
\[
\mathcal{E}_f(\gamma\tau)=\mathcal{E}_f(\tau) +\omega, \, \omega\in\Lambda_f,
\]
for $\gamma\in \Gamma_0(N)$. We then define the Weierstrass form by
\[
Z_f(\tau)=\zeta^*_{\Lambda_f}(\mathcal{E}_f(\tau)).
\]
The following theorem was proven in \cite{guerzhoy} and \cite{agor}.
\begin{Thm}\label{thm:maink=1}
Assume the notation and hypotheses above. The following are true:
\begin{enumerate}
\item The poles of the holomorphic part $Z_f^+(\tau)$ of $Z_f(\tau)$ are precisely those points for which $\mathcal{E}_f(\tau)\in\Lambda_f$.
\item If  $Z_f^+(\tau)$ has poles in $\HH$, then there is a canonical modular function $M_f(\tau)$ with algebraic coefficients on $\Gamma_0(N)$ for which  $Z_f^+(\tau)-M_f(\tau)$ is holomorphic on $\HH$.
\item The function $Z_f(\tau)-M_f(\tau)$ is a harmonic weak Maass form of weight $0$ on $\Gamma_0(N)$. 
\item We have that
\[
\xi_0( Z_f(\tau))= -\frac{2\pi i}{\mathrm{vol}(\Lambda_f)}f(\tau).
\]
\end{enumerate}
\end{Thm}

\section{Vector-valued Jacobi--Weierstrass forms}\label{sec:vvJW}
We first construct the Jacobi--Weierstrass $\zeta$-function in Section \ref{subsec:higher}, which is a higher degree analogue of the Weierstrass $\zeta$-function.  In Section \ref{subsec:completion} we construct its completion based on Rolen's approach and in Section \ref{subsec:vector} we define a vector-valued analogue.

\subsection{Elliptic functions of higher degree}\label{subsec:higher}

Generalizing \eqref{eq:ellipticthetasigma} we define a higher degree analogue of the Weierstrass $\sigma$-function. We let $\Lambda_\tau=\ZZ+\ZZ\tau$ for $\tau\in\HH$.
\begin{Def}	 Let $\vec{z}\in\CC^g$. We define the {\it Jacobi--Weierstrass $\sigma$-function} by
	 \begin{equation}\label{def_sigma}
	 \sigma_{\Lambda_\tau}(\vec{z})=e^{-u(\tau) Q(\vec{z})}\theta(\tau, \vec{z}),
	 \end{equation}
	 where $u(\tau)$ is a function in $\tau$ defined in \eqref{eq:u}. 
\end{Def}	 
Note that for $g=1$ the above definition differs from the classical definition of the Weierstrass $\sigma$-function by a (non-zero) constant factor depending on the lattice. This is however of no further importance here.

	 Let  $\vec{m}, \vec{n}\in \ZZ^g$. Using the properties of the Jacobi theta function from \eqref{lattice}, we directly see that
	\begin{equation}\label{sigma_transform}
	\sigma_{\Lambda_\tau}( \vec{z}+\vec{m}\tau+\vec{n})=e^{-u(\tau) (Q(\vec{z+\vec{m}\tau+\vec{n}})-Q(\vec{z}))} e^{-\pi i \tau (\vec{m},\vec{m})}e^{-2\pi i (\vec{z},\vec{m})}\sigma_{\Lambda_\tau}(\vec{z})
	\end{equation}
	for $\vec m,\vec n\in\ZZ^g$.
	We now define higher degree analogues of the $\wp$- and $\zeta$-function. 
	
	\begin{Def}		 
Let $\vec{z}\in\CC^g$.  We define the {\it Jacobi--Weierstrass $\zeta$-function} by
\begin{equation}
\zeta_{\Lambda_\tau,j}(\vec{z})=\frac{\partial}{\partial z_j} \log\sigma_{\Lambda_\tau,}(\vec{z}),\,1\leq j \leq g,
\end{equation}
and the {\it Jacobi--Weierstrass $\wp$-function} by
\begin{equation}
\wp_{ji}(\vec{z})=\wp_{\Lambda_\tau,ji}(\vec{z})=\frac{\partial}{\partial z_i}\zeta_{\Lambda_\tau,j}(\vec{z}), \,1\leq i,j\leq g.
\end{equation}
\end{Def}	
	
We first show that $\wp_{ji}$ is invariant under the lattice $\Lambda_\tau^g$.
\begin{Lem} For $1\leq i, j \leq g$, the Jacobi--Weierstrass $\wp$-function $\wp_{ji}$ is $\Lambda_{\tau}^{g}$-invariant, i.e.\ it holds that
\begin{equation}
	\wp_{ji}(\tau, \vec{z}+\vec{m}\tau+\vec{n})=\wp_{ji}(\tau, \vec{z}),
	\end{equation}
for $\vec{m},\vec{n}\in\ZZ^g$.	
\end{Lem}	
	
\begin{proof}
Let $\vec{m}=(m_1,\ldots,m_g),\,\vec{n}=(n_1,\ldots,n_g)\in\ZZ^g$. 
Using the transformation properties of the higher degree $\sigma$-function (compare \eqref{sigma_transform}) we see that
\[
	\frac{\left(\frac{\partial}{\partial z_j}\sigma_{\Lambda_\tau}\right)(\vec{z}+\vec{m}\tau+\vec{n})}{\sigma_{\Lambda_\tau}(\vec{z}+\vec{m}\tau+\vec{n})}
	=
	\frac{\left(\frac{\partial}{\partial z_j}\sigma_{\Lambda_\tau}\right)(\vec{z})}{\sigma_{\Lambda_\tau}(\vec{z})}
	+
	(-u(\tau)(G_L(\vec m\tau+\vec n))_j-2\pi i (G_L\vec m)_j),	\]
where, as earlier (see Section~\ref{sec:jacobi}), $G_L$ denotes the Gram matrix of the lattice $L$ used to define the theta series.
This implies that
	\begin{equation}\label{period}
	\zeta_{\Lambda_\tau,j}(\vec{z}+\vec{m}\tau+\vec{n}) - \zeta_{\Lambda_\tau,j}( \vec{z})=-u(\tau)(G_L(\vec m\tau+\vec n))_j-2\pi i (G_L\vec m)_j.
	\end{equation}
The right hand side of the equation above is independent of $\vec{z}$ and depends only on the lattice, thus
\[
	\frac{\partial}{\partial z_i}\zeta_{\Lambda_\tau,j}( \vec{z}+\vec{m}\tau+\vec{n}) - \frac{\partial}{\partial z_i}\zeta_{\Lambda_\tau,j}(\vec{z})=0.
	\]
This implies the desired invariance of $\wp_{ji}$.	
\end{proof}	
The independence of $\vec{z}$ of the difference in \eqref{period} justifies the following definition.	
\begin{Def} We define the {\it quasi-periods} to be
\[
	\eta_{jl, 1}(\tau)=\zeta_{\Lambda_\tau,j}(\vec{z}+e_l)-\zeta_{\Lambda_\tau,j}(\vec{z})=0, \, \text{if }j\neq l,
	\]
	\begin{equation}\label{eq:u}
	\eta_{jj, 1}(\tau)=\zeta_{\Lambda_\tau,j}( \vec{z}+e_j)-\zeta_{\Lambda_\tau,j}( \vec{z})=:u(\tau).
	\end{equation}
	Here, $e_i \in\ZZ^g$ denotes the $i$-th unit vector. 
	\end{Def}
	
\begin{rmk}
This generalizes the definition of the quasi-periods in the classical setting, where we have
\[
\eta_1(\tau)=\zeta_{\Lambda_\tau}(z+1)-\zeta_{\Lambda_\tau}(z).
\]
\end{rmk}

\subsection{The completion of the Jacobi--Weierstrass $\zeta$-function}\label{subsec:completion}
In this section we complete the Jacobi--Weierstrass $\zeta$-function such that it is lattice invariant. This generalizes the construction in \eqref{eq:logtheta}.  Let $\Lambda_\tau= \ZZ+\ZZ\tau$ for $\tau\in\HH$. We define
\begin{equation}\label{eqzetajacobi}
\zeta_{z_j}^*(\vec{z})=\zeta_{\Lambda_\tau,j}^*(\vec{z})= \frac{(Y_{+, z_j}^{g/2, 1/2L}\theta)(\tau,\vec{z})}{\theta(\tau,\vec{z})},\,\vec{z}\in \CC^g.
\end{equation}

\begin{Prop}\label{propzetaell}
The function $\zeta^*_{z_j}(\vec{z})$ is a (non-holomorphic) Jacobi form of weight $1$ and index $0$ for the Jacobi subgroup $\J_L$ (see Section \ref{sec:jacobi}). In particular, $\zeta^*_{z_j}(\vec{z})$ is invariant under the lattice $\Lambda_\tau^g$.
\end{Prop} 

\begin{proof}
We first note that to show that $\zeta^*_{z_j}(\vec{z})$ is a Jacobi form of weight $1$, it is enough to show the identity \eqref{jacobi}. To check that $\zeta^*_{z_j}(\vec{z})$ is elliptic, we first compute the partial derivative $\frac{\partial}{\partial z_j}$ of \eqref{lattice}, namely
\begin{align*}
\left(\frac{\partial}{\partial z_j}\theta\right)(\tau, \vec z+\vec{m}\tau+\vec{n})
&=
e^{-\pi i \tau (\vec{m}, \vec{m})}e^{-2\pi i (\vec{z}, \vec{m})}\left(\frac{\partial}{\partial z_j}\theta\right)(\tau, \vec{z})\\
&\quad -2\pi i (G_L\vec{m})_je^{-\pi i \tau (\vec{m}, \vec{m})}e^{-2\pi i (\vec{z}, \vec{m})}\theta(\tau, \vec{z}).
\end{align*}
Adding $2\pi i \frac{(\Imm(G_L(\vec{z}+\vec{m}\tau+\vec{n})))_j}{\Imm(\tau)}e^{-\pi i \tau (\vec{m}, \vec{m})}e^{-2\pi i (\vec{z}, \vec{m})}\theta(\tau, \vec{z})$ and using that $ \frac{\Imm(G_L(\vec{m}\tau+\vec{n}))}{\Imm(\tau)}=G_L\vec m, $
we get 
\begin{equation}
(Y_{+, z_j}^{g/2, 1/2L}\theta)(\tau,\vec{z}+\vec{m}\tau+\vec{n})
=
 e^{-\pi i \tau (\vec{m}, \vec{m})}e^{-2\pi i (\vec{z}, \vec{m})}(Y_{+, z_j}^{g/2, 1/2L}\theta)(\tau, \vec{z}).
\end{equation}
Dividing by $\theta(\tau,\vec{z}+\vec{m}\tau+\vec{n})
=
 e^{-\pi i \tau (\vec{m}, \vec{m})}e^{-2\pi i (\vec{z}, \vec{m})}\theta(\tau, \vec{z})$, we obtain
\[
\frac{(Y_{+, z_j}^{g/2, 1/2L}\theta)(\tau,\vec{z}+\vec{m} \tau+\vec{n})}{\theta(\tau,\vec{z}+\vec{m}\tau+\vec{n})}
=
\frac{(Y_{+, z_j}^{g/2, 1/2L}\theta)(\tau,\vec{z})}{\theta(\tau,\vec{z})}
\]
which finishes the proof that $\zeta^*_{z_j}(\vec{z})$ is elliptic.

Now we compute the action under the matrix $B=[\left(\begin{smallmatrix} a& b \\ c& d\end{smallmatrix}\right), \sqrt{c\tau+d}]\in \J_L$. It is enough to show identity \eqref{jacobi},which is equivalent to showing
\[
(c\tau+d)^{-g/2-1} e^{-\pi i c\frac{(\vec{z}, \vec{z})}{c\tau+d}}(Y_{+, z_j}^{g/2, 1/2L}\theta)\left( \frac{a\tau+b}{c\tau+d}, \frac{\vec{z}}{c\tau+d}\right) 
=
Y_{+, z_j}^{g/2, 1/2L}\theta,
\]
We compute the two components of the action of $Y_{+, z_j}^{g/2, 1/2L}=\frac{\partial}{\partial z_j}+2\pi i \frac{\Imm (G_L \vec{z})_j}{\Imm \tau} $. The transformation property  \eqref{jacobi_theta}  implies that
\[
\theta\left(\frac{a\tau+b}{c\tau+d}, \frac{\vec{z}}{c\tau+d}\right)=(c\tau+d)^{g/2}e^{\pi i c\frac{(\vec{z}, \vec{z})}{c\tau+d}}\theta(\tau, \vec{z}).
\]
We differentiate and obtain
\begin{align*}
\left(\frac{\partial}{\partial z_j}\theta\right)\left(\frac{a\tau+b}{c\tau+d}, \frac{\vec{z}}{c\tau+d}\right)&=(c\tau+d)^{g/2}2\pi i c (G_L\vec{z})_j e^{\pi i c\frac{(\vec{z}, \vec{z})}{c\tau+d}}\theta(\tau, \vec{z})\\
&\quad\quad+(c\tau+d)^{g/2+1}e^{\pi i c\frac{(\vec{z}, \vec{z})}{c\tau+d}}\left(\frac{\partial}{\partial z_j} \theta\right)(\tau, \vec{z}).
\end{align*}
Multiplying by $(c\tau+d)^{-g/2-1} e^{-\pi i c\frac{(\vec{z}, \vec{z})}{c\tau+d}}$ yields
	\begin{equation}\label{comp1}
	2\pi i c (G_L\vec{z})_j (c\tau+d)^{-1}\theta(\tau, \vec{z})+\left(\frac{\partial}{\partial z_j} \theta\right)(\tau, \vec{z}).
	\end{equation}

For the second part, we have to multiply the term 
\[
(c\tau+d)^{-g/2-1}e^{-\pi i c\frac{(\vec{z}, \vec{z})}{c\tau+d}}\theta \left(\frac{a\tau+b}{c\tau+d}, \frac{\vec{z}}{c\tau+d}\right)=(c\tau+d)^{-1}\theta (\tau, \vec{z})
\]
by $2\pi i (\Imm \frac{(G_L \vec{z})_j}{c\tau+d} )/(\Imm \frac{a\tau+b}{c\tau+d})=(c\tau+d)\frac{\Imm (G_L\vec{z})_j}{\Imm \tau}-c(G_L\vec{z})_j$, resulting in
\begin{equation}\label{comp2}
	2\pi i \frac{\Imm (G_L\vec{z})_j}{\Imm \tau} \theta(\tau, \vec{z})- 2\pi i c(G_L\vec{z})_j(c\tau+d)^{-1}\theta(\tau, \vec{z}).
\end{equation}
Adding \eqref{comp1} and \eqref{comp2} we get $Y_{+, z_j}^{g/2, 1/2L}\theta$, as desired. This finishes the proof.
\end{proof} 

We obtain a similar description of the completed Jacobi--Weierstrass $\zeta$-function as in the degree $1$ case.
\begin{Lem}\label{zeta_j_form}
The completed Jacobi--Weierstrass $\zeta$-function can be written as
\[
\zeta_{z_j}^*(\tau,\vec{z})=\zeta_{\Lambda_\tau,j}(\vec{z})-u(\tau)z_j+2\pi i\frac{\Imm((G_L\vec z) _j)}{\mathrm{Vol}(\Lambda_\tau)}.
\]
\end{Lem}
\begin{proof}
Using the definition of the weight raising operator $Y_{+, z_j}^{g/2, 1/2L}$ we find
\[
\frac{(Y_{+, z_j}^{g/2, 1/2L}\theta)(\tau,\vec{z})}{\theta(\tau,\vec{z})}
=
\frac{(\frac{\partial}{\partial z_j}\theta)(\tau,\vec{z})}{\theta(\tau,\vec{z})}+2\pi i \frac{\Imm((G_L\vec z)_j)}{\Imm(\tau_j)}.
\]
In order to compute the logarithmic derivative of the Jacobi theta-function, we note that
\[
\frac{\partial}{\partial z_j}\sigma_{\Lambda_\tau}(\vec{z})=c(\tau)e^{-u(\tau) Q(\vec{z})}\left(\frac{\partial}{\partial z_j}\theta\right)(\tau, \vec{z})-z_ju(\tau)c(\tau)e^{-u(\tau) Q(\vec{z})}\theta(\tau, \vec{z}),
\]
which implies
\[
\frac{\frac{\partial}{\partial z_j}\sigma_{\Lambda_\tau}(\vec{z})}{\sigma_{\Lambda_\tau}(\vec{z})}=e^{-u(\tau) Q(\vec{z})}\frac{(\frac{\partial}{\partial z_j}\theta)(\tau, \vec{z})}{\sigma(\tau, \vec{z})}-z_ju(\tau)e^{-u(\tau) Q(\vec{z})}\frac{\theta(\tau, \vec{z})}{\sigma(\tau, \vec{z})}.
\]
Thus, using the definition \eqref{def_sigma} of the Jacobi-Weierstrass $\sigma$-function,  we find
\begin{equation}\label{eq:quasi2}
\frac{\frac{\partial}{\partial z_j} \theta(\tau, \vec{z})}{\theta(\tau, \vec{z})}=\zeta_{\Lambda_\tau,j}(\vec{z})+z_ju(\tau).
\end{equation}
\end{proof}


\subsection{Vector-valued Jacobi--Weierstrass forms}\label{subsec:vector}
We now fix $g=k-2$ and introduce vector-valued Jacobi--Weierstrass forms. Let $z\in\CC^{k-1}$ and $M\in\SL_2(\ZZ)$. Let $(e_1,e_2)$ denote the standard basis of $\CC^2$. 
We recall the notation from Section \ref{sec:changeofbasis} and write $(\vec{z})_{(e_1,e_2)}$ for the element in $\Sym^{k-2}(\CC^2)$ corresponding to $\vec{z}$.
\begin{Def} \label{def:jacobiwform}
Let the notation be as above.
We define the vector-valued {\it Jacobi--Weierstrass form} by
\[
(\widehat{\zeta}_{\vec{z}}(\vec{z}))_{(X_1, X_2)}=\sum_{j=0}^{k-2}\zeta^*_{z_j}(\vec{z}) c_j X_1^{j} X_2^{k-2-j}
\]
with $c_j=\binom{k-2}j$ as before, and $(X_1, X_2)=M\circ (e_1, e_2)$.
\end{Def}

Under the isomorphism $\CC^{k-1}\simeq \Sym^{k-1}(\CC^2)$ coming from the standard basis, we can write $\widehat{\zeta}_{\vec{z}}$ as $\widehat{\zeta}_{\vec{z}}((\vec{z})_{(e_1, e_2)})_{(X_1, X_2)}$. It defines a function $\widehat{\zeta}: \,\Sym^{k-2}(\CC^2)\times\SL_2(\ZZ)\to \Sym^{k-2}(\CC^2)$, where $\widehat{\zeta}((\vec{z})_{(e_1, e_2)}, M)=\widehat{\zeta}_{\vec{z}}((\vec{z})_{(e_1, e_2)})_{M\circ (e_1, e_2)}$. In the definition we suppressed the dependence on the lattice $\Lambda_\tau=\ZZ+\ZZ\tau$, $\tau\in\HH$. Fixing a basis $(X_1, X_2)$, we have the following result that is an immediate consequence of Proposition \ref{propzetaell}
\begin{Thm}\label{jacobi_thm} The function $(\widehat{\zeta}_{\vec{z}})_{(X_1, X_2)}$ is a (non-holomorphic) Jacobi form of weight $1$ and index $0$, invariant under the lattice $\Lambda_{\tau}^{k-1}$. 
\end{Thm}

\begin{rmk}
Note that there are two sets of coordinates involved in the definition of the Jacobi--Weierstrass form. First we have the element in $\Sym^{k-2}(\CC^2)$ corresponding to $\vec{z}\in\CC^{k-2}$ in terms of the standard basis. Second, the Jacobi--Weierstrass function itself is an element of $\Sym^{k-2}(\CC^2)$ in terms of the basis $(X_1,X_2)$.
\end{rmk}

We now consider directional derivatives of the Jacobi--Weierstrass form. Recall that
\[
(\vec{z})^{\mathrm{t}}_{(e_1,e_2)}=\sum_{\ell=0}^{k-2} z_\ell c_\ell e_1^{\ell}e_2^{k-2-\ell}.
\]
We then define the directional derivative in the direction $N(M)(\vec{z})^{\mathrm{t}}_{(e_1,e_2)}$ by
\[
\zeta_{(N(M) \vec{z})_i}(\vec{z})=\frac{\frac{\partial}{\partial (N(M) \vec{z})_{i}} \theta(\tau, \vec{z})}{\theta(\tau, \vec{z})}.
\]
Here, $N(M)$ is the matrix defined in Lemma \ref{n(M)} and we write $(N(M)\vec{z})_i$ for the $i^{th}$ term in the vector $N(M)\vec{z}$. We consider the global Jacobi--Weierstrass $\zeta$-function
\begin{equation}
\left(\zeta_{M^{-1}}(\vec{z})\right)_{(e_1, e_2)}
=
\sum_{\ell=0}^{k-2} \zeta_{ (N(M)\vec{z})_\ell}(\vec{z}) c_\ell e_1^\ell e_2^{k-2-\ell},
\end{equation}
as well as the global completed Jacobi--Weierstrass $\zeta$-function
\begin{equation}
\left(\widehat{\zeta}_{M^{-1}}(\vec{z})\right)_{(e_1, e_2)}
=
\sum_{\ell=0}^{k-2} \zeta^*_{ (N(M)\vec{z})_\ell}(\vec{z}) c_\ell e_1^\ell e_2^{k-2-\ell}.
\end{equation}
We occasionally use the notation $\zeta_{N(M)\vec{z}}$ and $\widehat{\zeta}_{N(M)\vec{z}}$ when we want to emphasize the indices. To shorten notation we will also write $\widehat{\zeta}_{M^{-1}}(\vec{z})$ for $\left(\widehat{\zeta}_{M^{-1}}(\vec{z})\right)_{(e_1, e_2)}$ in some places.

We note that in this notation $\widehat{\zeta}_{\vec{z}}(\vec{z})=\widehat{\zeta}_{Id}(\vec{z})$ and we show that the global directional derivative corresponds to a change of coordinates of $\widehat{\zeta}_{\vec{z}}$.
\begin{Lem}\label{global_change} 
For $M\in \SL_2(\ZZ)$, $\vec{z}\in \CC^{k-1}$ and $(X_1, X_2)$ a basis of $\CC^2$, we have
\[
\left( \widehat{\zeta}_{M}(\vec{z})\right)_{(X_1, X_2)}
=
\left(\widehat{\zeta}_{\vec{z}}(\vec{z})\right)_{M\circ(X_1, X_2)}.
\]
\end{Lem}
\begin{proof}
We first consider the non-completed version of Definition \ref{def:jacobiwform}, i.e.\
\[
 \left(\zeta_{M}( \vec{z})\right)_{(X_1, X_2)}=\sum_{\ell=0}^{k-2} \zeta_{(N(M^{-1}) \vec{z})_\ell}(\vec{z}) c_\ell X_1^\ell X_2^{k-2-\ell}.
\]
By definition this equals (in vector notation)
\[
\frac{1}{\theta(\tau, \vec{z})}
\left(\frac{\partial}{\partial (N(M^{-1})\vec{z})_0}\theta(\tau, \vec{z}), \dots, \frac{\partial}{\partial (N(M^{-1})\vec{z})_{k-2}}\theta(\tau, \vec{z}) \right)^{\mathrm{t}}_{(X_1, X_2)}.
\]
A short calculation using multivariable calculus yields
\begin{align*}
&\left(\frac{\partial}{\partial (N(M^{-1})\vec{z})_0}\theta(\tau, \vec{z}), \dots, \frac{\partial}{\partial (N(M^{-1})\vec{z})_{k-2}}\theta(\tau, \vec{z}) \right)^{\mathrm{t}}_{(X_1, X_2)} 
\\
&\quad =N(M) \left(\frac{\partial}{\partial z_0}\theta(\tau, \vec{z}), \dots,\frac{\partial}{\partial z_{k-2}}\theta(\tau, \vec{z})\right)^{\mathrm{t}}_{(X_1, X_2)}.
\end{align*}
Therefore, we have
\[
  \left(\zeta_{M}( \vec{z})\right)_{(X_1, X_2)} =N(M) \left(\zeta_{ \vec{z}}(\vec{z})\right)_{(X_1, X_2)}.
\]
Plugging this into the completion of the Jacobi--Weierstrass function yields the desired result.
\end{proof}


\section{Vector-valued Jacobi--Weierstrass as polar harmonic weak Maass forms}\label{vector_valued}

In this section we define the higher weight analogue of Weierstrass harmonic weak Maass forms.  Throughout, we let $f$ be a newform of weight $k$ for $\Gamma$ with rational Fourier coefficients, with $\Gamma_1(N)\subset \Gamma$ for some integer $N$. We will use the Eichler integral $\E_f(\tau,(X_1,X_2))$ and the lattice $\Lambda_f$ associated to $f$ from Section \ref{sec:eichlerint}.

By $V$ we denote the infinite-dimensional vector space 
\[
V=\mathrm{Functions}\{\SL_2(\ZZ) \ra \Sym^{k-2}(\CC^2)\}.
\] 
The action of $\Gamma$ on $V$ is given by
\begin{equation}\label{rho}
\rho(\gamma) F(M) =\gamma \circ F(\gamma^{-1} M), 
\end{equation}
where $\gamma \in \Gamma$. Here $\Gamma$ acts on $\SL_2(\ZZ)$ by matrix multiplication and on $\Sym^{k-2}(\CC^2)$ by acting on the basis as described in Section \ref{sec:changeofbasis}.

\begin{Def}
Let $\tau\in\HH$ and $M\in\SL_2(\ZZ)$. 
We define the \textit{vector-valued Jacobi--Weierstrass form} by
\begin{equation}\label{eqdefF}
F(\tau):=
\left[ M\mapsto \left(\widehat{\zeta}_{M}(\E_f(\tau, M^{-1}\circ (e_1, e_2)))_{(e_1, e_2)}\right)\right].
\end{equation}
\end{Def}

We write 
\[
F(\tau, M)=\widehat{\zeta}_{M}(\Lambda_f, \E_f(\tau, M^{-1}\circ (e_1, e_2)))_{(e_1, e_2)},
\]
for the image of $F$ in $ \Sym^{k-2}(\CC^2)$ and drop the dependence on $\Lambda_f$ in the definition of $\widehat{\zeta}$.

\begin{rmk}
The function $F$ depends on the chosen basis of $ \Sym^{k-2}(\CC^2)$ in two places. We note two important changes of coordinates. By Lemma \ref{global_change} we see
\begin{equation}\label{def_F}
F(\tau, M)
=
\left(\widehat{\zeta}_{\vec{z}}\left( \E_f\left(\tau, M^{-1}\circ \left(e_1, e_2\right)\right)\right)\right)_{M\circ (e_1, e_2)}.
\end{equation}
This can also be taken as the definition of the function $F(\tau, M)$.

Moreover, we have a second change of coordinates inside the function given by changing the basis for the Eichler integral. As by definition we have $\E_f\left(\tau, M^{-1}\circ\left(e_1, e_2\right)\right)=N(M^{-1})\E_f(\tau, (e_1, e_2))
$ it follows that
\begin{equation}
F(\tau, M)
=
\left(\widehat{\zeta}_{M}\left(N(M^{-1})\E_f\left(\tau, \left(e_1, e_2\right)\right)\right)\right)_{(e_1, e_2)}.
\end{equation}

\end{rmk}

We now prove that $F$ defines a vector-valued polar harmonic weak Maass form of weight $0$ that maps to $f$ with values in a coefficient system under $\xi_0$.

\begin{Thm}\label{thm:main}
The function $F(\tau)$ is a vector-valued polar harmonic weak Maass form of weight $0$ for $\Gamma$ with respect to the representation $\rho$ and its image under $\xi_0$ is given by
\[
G(\tau)=\frac{2\pi i}{\mathrm{Vol}(\Lambda_f)} f(\tau) \ G_L\cdot (\tau e_1+e_2)^{k-2}.
\]
\end{Thm}

Here $G_L$ is the Gram matrix corresponding to the lattice $L$ which is implicitly used in the construction of the Jacobi--Weierstrass $\zeta$-function $\widehat{\zeta}$. For $\vec{z}=(z_0, \dots, z_{k-2})^t$ written in the basis $(X_1, X_2)$ of $\CC^2$, the action of $G_L=(g_{ij})_{0\leq i, j \leq k-2}$ is given by
\begin{equation}\label{action_G_L}
G_L \cdot (\vec{z})_{(X_1, X_2)}=\sum_{\ell=0}^{k-2} \left(\sum_{i=0}^{k-2} g_{\ell i} z_i \right)c_\ell X_1^{\ell} X_2^{k-2-\ell}
=
\sum_{i=0}^{k-2} z_i \left(\sum_{\ell=0}^{k-2} g_{\ell i}  c_\ell X_1^{\ell} X_2^{k-2-\ell}\right).
\end{equation}

We note that for the lattice $\ZZ^{k-1}$ with the standard inner product, the theorem simplifies to
\[
G(\tau)=\frac{2\pi i}{\mathrm{Vol}(\Lambda_f)} f(\tau)(\tau e_1+e_2)^{k-2}.
\]

\begin{rmk}
\begin{enumerate}
\item The function $G:\HH \ra \Sym^{k-2}(\CC^2)$ is a holomorphic vector-valued modular form of weight $2$.  
\item For the special case of weight $k=2$ it is possible to construct a modular function $M$ with algebraic Fourier coefficients which eliminates all poles of $F$ on the upper half-plane (see \cite{agor} resp.\ Theorem \ref{thm:maink=1}). Then, $F+M$ is a harmonic weak Maass forms whose shadow is the newform $f$. The proof given in loc.\ cit.\ uses the proof for the well-known fact that $j(\tau)$ and $j(N\tau)$ generate the field of modular functions for the group $\Gamma_0(N)$ (see e\ g.\ \cite[Theorem 11.9]{cox}). To the authors' knowledge, there is, however, no systematic theory of modular functions with respect to the (infinite-dimensional) representation $\rho$ from \eqref{rho} which would allow to prove an analogous result in general.
\end{enumerate}
\end{rmk}

We prove Theorem \ref{thm:main} in several steps. We first show the $\Gamma$-invariance of each one of the terms $\widehat{\zeta}_{M}\left( \E_f\left(\tau, M^{-1}\circ\left(e_1, e_2\right)\right)\right)$.

\begin{Lem}\label{zeta_inv} 
Let the notation be as above. For $\gamma\in \Gamma$, we have
\[
\widehat{\zeta}_{M}\left(\E_f\left(\tau, (\gamma^{-1} M)^{-1}\circ\left(e_1, e_2\right)\right)\right)
=
\widehat{\zeta}_{M}\left(\E_f\left(\gamma\tau, M^{-1} \circ\left(e_1, e_2\right)\right)\right).
\]
\end{Lem}
\begin{proof}
We write $(X_1, X_2)=M^{-1}\circ (e_1, e_2)$ and note that
\[
(\gamma^{-1}M)^{-1} \circ\left(e_1, e_2\right)=\left(e_1, e_2\right)M^{-1}\gamma = \left(X_1, X_2\right)\gamma.
\]
By definition of the change of basis matrix $N(\cdot)$, for the Eichler integral we have 
  \[
  \E_f\left(\tau, \gamma \circ\left(X_1, X_2\right)\right)
=
N(\gamma) \E_f(\tau, (X_1, X_2))\quad \text{and}\quad \E_f(\tau, (X_1, X_2))=N(M^{-1})\E_f(\tau, (e_1, e_2)).
\]
Moreover, from Lemma \ref{gamma_inv} we have
\begin{equation}
\E_f(\gamma\tau, (X_1, X_2))=\E_f(\tau, \gamma\circ(X_1, X_2))-\omega,
\end{equation}
where $\omega\in \Lambda_f^{k-1}$ with $\Lambda_f^{k-1}$ written in the basis $\gamma\circ (X_1, X_2)=(e_1, e_2)M^{-1}\gamma$, thus it is an element of $N(M^{-1}\gamma)\Lambda_f^{k-1}$.
By Lemma \ref{inv_lambda} we know that $N(M^{-1}\gamma)\Lambda_f^{k-1}=\Lambda_f^{k-1}$. Therefore we obtain
\[
\E_f(\gamma \tau, M^{-1}\circ (e_1, e_2))-
\E_f(\tau, (\gamma^{-1} M)^{-1}\circ (e_1, e_2)) \in \Lambda_f^{k-1}.
\]
Since $\widehat{\zeta}_{M}$ is $\Lambda_f^{k-1}$-invariant from Proposition \ref {propzetaell}, this gives the result.
\end{proof}

\begin{proof}[Proof of Theorem \ref{thm:main}]
We first prove that $F$ is invariant under the action of $\Gamma$. The harmonicity (away from possible poles) will follow from the fact that $\xi_0(F)$ is a holomorphic cusp form (with values in a coefficient system), since the Laplacian equals $\Delta_k=-\xi_{2-k} \circ \xi_{k} $.

Let $\gamma\in\Gamma$. For $M\in \SL_2(\ZZ)$ and $\tau\in \HH$, we need to show that 
\[
F(\gamma \tau, M)=\rho(\gamma) F(\tau, M),
\]
where the action $\rho:\Gamma \ra\End(V)$ was defined in \eqref{rho}, thus
\[
\rho(\gamma) F(\tau, M)
=
\gamma\circ\left(\widehat{\zeta}_{(\gamma^{-1} M)}\left(\E_f\left(\tau,\left((\gamma^{-1} M)^{-1}\right)\circ\left(e_1, e_2\right)\right)\right)\right)_{(e_1, e_2)}.
\] 
By Lemma \ref{global_change} this equals
\[
\gamma\circ \widehat{\zeta}_{\vec{z}}\left(\E_f(\tau,\left(\gamma^{-1} M\right)^{-1}\circ \left(e_1, e_2\right))\right)_{(e_1, e_2) \gamma^{-1}M}
\]
and this expression simplifies to 
\[\widehat{\zeta}_{\vec{z}}\left(\E_f(\tau,\left( \gamma^{-1} M\right)^{-1}\circ \left(e_1, e_2\right))\right)_{(e_1, e_2) \gamma\gamma^{-1}M}
=
\widehat{\zeta}_{M}\left(\E_f(\tau,\left(\gamma^{-1} M\right)^{-1}\circ \left(e_1, e_2\right))\right)_{(e_1, e_2)}.
\]
Finally, we use Lemma \ref{zeta_inv} and obtain that 
\[
\rho(\gamma) F(\tau, M)
=
\widehat{\zeta}_{M}\left(\E_f\left(\gamma \tau, M^{-1}\circ \left(e_1, e_2\right) \right)\right)_{(e_1, e_2)} 
=
F(\gamma \tau, M).
\]
\smallskip

Now we will compute the image of $F(\tau,M)$ under $\xi_0=-2i \overline{\frac{\partial}{\partial \overline{\tau}}}$. This gives a function
\[
g(\tau, M)=\xi_0 F(\tau, M),
\]
with $g:\HH \ra V$. We will compute $g$ explicitly and show that it is indeed constant on $\SL_2(\ZZ)$. We write 
\begin{equation}\label{Ycoord}
F(\tau, M)
=
\left(\widehat{\zeta}_{\vec{z}}\left(\E_f(\tau, M^{-1}\circ\left(e_1, e_2\right)\right)  \right)_{(Y_1, Y_2)},
\end{equation}
with $(Y_1,Y_2)=M\circ(e_1, e_2)$ and consider each of the coordinates $\zeta_{z_\ell}^*$ in the basis $(Y_1, Y_2)$. By Lemma \ref{zeta_j_form} we obtain 
\[
\zeta_{z_\ell}^*(\vec{z})=\zeta_{z_\ell}(\vec{z})-u(\tau_f)z_\ell+2\pi i\frac{\Imm((G_L\vec z)_\ell)}{\Imm(\tau_f)}.
\]
Writing $\Imm(z)=(z-\overline{z})/2i$ and plugging in $\vec{z}=\E_f(\tau, M^{-1}\circ (e_1, e_2)))$ we see that 
\[
\overline{\frac{\partial}{\partial \overline{\tau}} \zeta_{z_\ell}^*\left(\E_f\left(\tau, M^{-1}\circ\left(e_1, e_2\right)\right)\right)}
=
-\frac{\pi}{\Imm(\tau_f)}
\overline{\left(G_L \frac{\partial}{\partial\overline{\tau}}\overline{\E_f\left(\tau, M^{-1}\circ\left(e_1, e_2\right)\right)}\right)}_{\ell}.
\]
As $G_L$ accounts for changing the basis from \eqref{action_G_L} 
\[
G_L \cdot (\vec{z})_{(Y_1, Y_2)}
=
\sum_{i=0}^{k-2} z_i \left(\sum_{\ell=0}^{k-2} g_{\ell i}  c_\ell Y_1^{\ell} Y_2^{k-2-\ell}\right),
\]
it is enough to compute each component $\overline{\frac{\partial}{\partial\overline{\tau}}\left(\overline{\E_f\left(\tau, M^{-1}\circ\left(e_1, e_2\right)\right)}\right)}_{\ell}$, where
\begin{align*}
 \left(\E_f\left(\tau, M^{-1}\circ\left(e_1, e_2\right)\right)\right)_\ell= \int_\tau^\infty f(t) (d t - b)^\ell (-c t+a)^{k-2-\ell} dt,
\end{align*}
with $M=\left(\begin{smallmatrix} a & b\\ c & d\end{smallmatrix}\right)\in\SL_2(\ZZ)$. This gives us 
\[
\overline{\frac{\partial}{\overline{\partial \tau}}\overline{\left(\E_f\left(\tau, M^{-1}\circ\left(e_1, e_2\right)\right)\right)_\ell}}= f(\tau) (d \tau -b )^\ell (-c\tau+a)^{k-2-\ell}
\]
and with the action of $G_L$ we get
\begin{align*}
\xi_0 F(\tau, M) &= \frac{2\pi i}{\Imm(\tau_f)}G_L\cdot \sum_{\ell=0}^{k-2} f(\tau) (d\tau-b)^\ell (-c\tau +a)^{k-2-\ell} c_\ell Y_1^\ell Y_2^{k-2-\ell}\\
&=  \frac{2\pi i}{\Imm(\tau_f)}f(\tau)  \ G_L\cdot\left((dY_1-cY_2)\tau+(-bY_1+aY_2)^{k-2}\right)\\
&=  \frac{2\pi i}{\Imm(\tau_f)}f(\tau)  \ G_L\cdot\left(M^{-1}\circ (Y_1\tau+Y_2)^{k-2}\right)\\
&=  \frac{2\pi i}{\Imm(\tau_f)} f(\tau) \ \left(G_L\cdot (e_1\tau+e_2)^{k-2}\right).
\end{align*}
The last expression is independent of $M$. The holomorphicity is obvious and the modularity follows from \eqref{eq:actiononsymandtau}. Moreover, we see that $F(\tau)$ is harmonic.

\medskip

It remains to show that  $F$ grows at most linear exponentially when $\tau$ approaches a cusp of $\Gamma$.  For this let $\sigma\in\SL_2(\ZZ)$ be arbitrary.  The same computation as in the proof of Lemma~\ref{action_gamma} shows that 
\[
\E_f(\sigma\tau,M^{-1}\circ (e_1,e_2))=\E_{f|\sigma}(\tau,M^{-1}\sigma\circ (e_1,e_2))+C
\]
for a certain constant vector $C$. Since $f$ is a cusp form and $f|\sigma$ yields the Fourier expansion of $f$ at the cusp $\sigma\infty$ by definition, it vanishes exponentially as $\tau\to\infty$.  Since $\theta$ is holomorphic in all elliptic variables,  it follows that the partial logarithmic derivatives of $\theta$ all have at most a finite order pole at the divisor of $\theta$. But this implies that, as $\tau\to\infty$, the analytic part of $F(\sigma.\tau,M)$ grows at most linear exponentially if $C$ lies in the divisor of $\theta$ and is bounded otherwise. This shows the claim, as the non-analytic part is clearly bounded.
\end{proof}

\section{Laurent expansion and poles}\label{sec:laurent}
In the definition of the Jacobi theta function we allowed an arbitrary positive definite integral bilinear form on $\ZZ^g$. Especially when computing examples (see Section \ref{examples}), it is most convenient to take  the standard bilinear form, but allow for a non-trivial characteristic. In what follows we define
\begin{gather}
\theta(\tau,\vec{z})=\sum_{n\in \left(\frac 12+\ZZ\right)}e^{\pi i \left(\sum_i n_i^2\right)\tau +2\pi i\sum_i n_iz_i}=\prod_{i=1}^g \vartheta(\tau,z_i).
\end{gather}
In this setting we immediately find from \eqref{eqzetajacobi} that 
$$\zeta_{z_j}^*(\tau,\vec{z})=\frac{Y_{+,z_j}^{g/2,1/2}\theta(\tau,\vec{z})}{\theta(\tau,\vec{z})}=\frac{Y_{+,z_j}^{1/2,1/2}\vartheta(\tau,z_j)}{\vartheta(\tau,z_j)}=\zeta_{\Lambda_\tau}^*(z_j),$$
where $\zeta^*$ here denotes the completion of the usual $g=1$ Weierstrass $\zeta$-function.

Similarly, the directional derivatives occurring in the definition of the function $F$ in \eqref{eqdefF} can be expressed explicitly as
$$\widehat{\zeta}_{(N(M^{-1})\vec{z})_j}(\tau,\vec{z})=\sum_i N(M)_{ij}\zeta_{\Lambda_\tau}^*(z_i),$$
which follows directly from the definition of the directional derivative, where $N(M)=\{N(M)_{ij}\}_{0\leq i, j \leq k-2}$.

We recall the Laurent expansion of the completed Weierstrass $\zeta$-function
\[
\zeta_{\Lambda_\tau}^*( z)=\frac{1}{z}-\sum_{n\geq 1} G_{2n+2}(\Lambda_{\tau}) z^{2n+1}-S(\Lambda_{\tau})z-\frac{\pi}{\Imm(\tau)} \overline{z},
\]
where $G_{2n}(\Lambda)=\sum_{\omega\in \Lambda\setminus\{0\}}\omega^{-2n}$ is the classical Eisenstein series of weight $2n$ and $ S(\Lambda_{\tau})=\lim\limits_{s\ra 0^+} \sum_{\omega\in \Lambda_{\tau}\setminus\{0\}}\frac{1}{\omega^2|\omega|^2}$. 

For the Laurent expansion of the Jacobi--Weierstrass $\zeta$-function 
\[
\widehat{\zeta}_{\vec{z}}(\tau, \vec{z}) =
\sum\limits_{\ell=0}^{k-2} \zeta^*_{\Lambda_{\tau}}(z_\ell)e_1^{\ell} e_2^{k-2-\ell}c_\ell
\]
we then obtain the following result

\begin{Lem} With the standard choice of lattice, the Laurent expansion of the Jacobi--Weierstrass $\zeta$-function $\zeta_{\vec{z}}$ is given by
\[
\zeta_{\vec{z}}(\tau, \vec{z})
=  
\sum_{\ell=0}^{k-2} \frac{1}{z_\ell} e_1^{\ell} e_2^{k-2-\ell}c_\ell +\sum_{n\geq 1} G_{2n+2}(\Lambda_{\tau})\sum_{\ell=0}^{k-2} z_\ell^{2n+1}e_1^{\ell} e_2^{k-2-\ell}c_\ell.
\]
Moreover, the Laurent expansion of its completion is given by
\[
\widehat{\zeta}_{\vec{z}}(\tau, \vec{z})
=
\zeta_{\vec{z}}(\tau, \vec{z})-S(\Lambda_{\tau}) \sum_{\ell=0}^{k-2} z_\ell e_1^{\ell} e_2^{k-2-\ell}c_\ell -\frac{\pi}{\Imm(\tau)}\overline{z_\ell} e_1^{\ell} e_2^{k-2-\ell}c_\ell.
\]

\end{Lem}

We note that when plugging in $\E_f(\tau, (e_1, e_2))$, we get the poles of the vector-valued function $F(\tau, Id):\HH \ra \Sym^{k-2}(\CC^2)$, for $\tau\in \HH$ such that
\[
\E_\ell(\tau, (e_1, e_2))\in \Lambda_f .
\]

The theory easily extends to general $M\in \SL_2(\ZZ)$ for
\[
\zeta_{M^{-1}}(\Lambda_{\tau}, N(M^{-1})\vec{z})
=
\sum\limits_{\ell=0}^{k-2} \zeta_{\Lambda_{\tau}}((N(M^{-1})\vec{z})_\ell) X_1^{\ell} X_2^{k-2-\ell}c_\ell,
\]
where $(X_1, X_2)=M\circ(e_1, e_2)$, and its completed version.

\begin{Lem} The Laurent expansion of the Jacobi--Weierstrass $\zeta$-function $\zeta_{M^{-1}}(\Lambda, (\vec{z})_{M^{-1}})$ is given by 
\[
\zeta_{M^{-1}}(\Lambda, (\vec{z})_{M^{-1}})
=
\sum_{\ell=0}^{k-2} \frac{1}{(\vec{z})_{M^{-1}, \ell}} X_1^{\ell} X_2^{k-2-\ell}c_\ell +\sum_{k\geq 0} G_{k-2}(\Lambda) \sum_{\ell=0}^{k-2} (\vec{z})_{M^{-1}, \ell}^{k-1}X_1^{\ell} X_2^{k-2-\ell}c_\ell,
\]
where $(X_1 ,X_2)=M\circ(e_1, e_2) $, and $(\vec{z})_{M^{-1}, \ell}=(N(M^{-1})(\vec{z}))_l$ is the $\ell^{th}$ component of the vector $N(M^{-1})(\vec{z})$. We get a similar expansion for the completion 
\[
\widehat{\zeta}_{M^{-1}}(\Lambda, (\vec{z})_{M^{-1}})=\zeta_{M^{-1}}(\Lambda, (\vec{z})_{M^{-1}})+S(\Lambda_{\tau}) \sum_{\ell=0}^{k-2} z_\ell X_1^{\ell} X_2^{k-2-\ell}c_\ell -\frac{\pi}{\Imm(\tau)}\overline{z_\ell} X_1^{\ell} X_2^{k-2-\ell}c_\ell.
\]
\end{Lem}

We recall that $
F(\tau, M)=\widehat{\zeta}_{M^{-1}}\left(\Lambda_f, N(M^{-1})\E_f\left(\tau, \left(e_1, e_2\right)\right)\right)$. Therefore, we can compute the poles of the Jacobi--Weierstrass  $\zeta$-function for the standard choice of symmetric form on $\ZZ^{k-1}$ explicitly.

\begin{Prop}\label{poles} Under the conditions above, the polar harmonic weak Maass form $F:\HH \ra V$ has poles of order $1$ at the values $\tau\in \HH$ for which there exists $M\in \SL_2(\ZZ)$ and $0\leq \ell\leq k-1$ such that
\[
(N(M^{-1})\E_f(\tau, (e_1, e_2)))_\ell
\in \Lambda_f. 
\]
\end{Prop}

\section{Examples}\label{examples}

We consider the setting of the previous section. In particular we fix the choice of the lattice $L$ in the definition of the Jacobi--Weierstrass $\zeta$-function to be the standard lattice $\ZZ^{k-1}$ with the standard bilinear form. We also just evaluate the image of the function $F$, plugging in the identity matrix.

\subsection{A Weierstrass form for $\Delta$}
Consider the Ramanujan $\Delta$-function,
$$\Delta(\tau):=q\prod_{n=1}^\infty (1-q^n)^{24}=\sum_{n\geq 1}\tau(n)q^n.$$
Using the built-in functions for modular symbols in Pari/Gp \cite{PARI98}, we can compute the period lattice of $\Delta$. Since $\SL_2(\ZZ)$ is generated by $T=\left(\begin{smallmatrix} 1 & 1 \\ 0 & 1\end{smallmatrix}\right)$ and $S=\left(\begin{smallmatrix} 0 & -1 \\ 1 & 0  \end{smallmatrix}\right)$, the period lattice is generated by the coefficients of the period polynomial $\int_0^\infty \Delta(t)(e_1+te_2)^{10} dt$, which equals
\begin{align*}
& \alpha \left(e_1^{10}- \frac{691}{1620}\binom{10}8 e_1^8e_2^2 + \frac{691}{2520}\binom{10}6 e_1^6e_2^4-\frac{691}{2520}\binom{10}4e_1^4 e_2^6 + \frac{691}{1620}\binom{10}2 e_1^2e_2^8-e_2^{10}\right)\\
 &\qquad\qquad\qquad +\beta\left(\binom{10}9 e_1^9e_2-\frac{25}{48}\binom{10}7 e_1^7e_2^3 +\frac 5{12}\binom{10}5 e_1^5e_2^5-\frac{25}{48}\binom{10}3e_1^3e_2^7 +\binom{10}1 e_1e_2^9\right).
 \end{align*}
with 
$$\alpha=0.00595896\ldots i\quad\text{and}\quad \beta=0.00370771\ldots.$$
Therefore we see that the period lattice for $\Delta$ is given by
$$\Lambda_\Delta=\omega_1\ZZ\oplus \omega_2\ZZ$$
with $\omega_1=\frac 1{48}\beta=7.7243968...\cdot 10^{-7}$ and $\omega_2=2.6274096...\cdot 10^{-7}i$.

With this data we can compute the Laurent expansion of the completed function $\widehat \zeta_{\Lambda_\Delta}(z)$ to be approximately
\begin{align*}
\widehat \zeta_{\Lambda_\Delta}(z)&=z^{-1} + 0.0016910\ldots z - 454230029641788589613076734.309657\ldots z^3+ O(z^5)\\
&\qquad\qquad\qquad-154795208574.9957812\ldots \overline z.
\end{align*}
For illustration purposes, we only consider the Eichler integral 
$$\mathcal E_{\Delta,0}(\tau)=\frac{1}{2\pi}(q - 12q^2 + 84q^3 - 368q^4 + 966q^5 - 1008q^6)+O(q^7)$$ 
since it has a proper Fourier expansion.  Plugging this into the holomorphic part of the Laurent expansion above yields the $0$-th component of the vector-valued harmonic weak Maass form $F$ for the $\Delta$-function (or rather its holomorphic part),
\begin{align*}
&\pi\left(q^{-1} + 12 + 60.0000428\ldots q + 79.999485\ldots q^2 \right.
\\
&\qquad\qquad\qquad \left.  - 291444838990458826940712.635494\ldots q^3+O(q^4)\right).
\end{align*}
We note that the Fourier coefficients grow very quickly. For instance, the coefficient of $q^{10}$ is approximately $1.33163\cdot 10^{61}.$

Since the Eichler integrals $\mathcal E_{\Delta,\ell}(\tau)$ can be evaluated quite efficiently, we can illustrate the modularity of the function $F$ we constructed by giving numerical evaluations of it and acting with elements of $\SL_2(\ZZ)$ on them.

For example let $\tau=2i$ and $\gamma=\left(\begin{smallmatrix} 2 & 5 \\ 1 & 3\end{smallmatrix}\right)$. 
In the standard basis $\{\binom{10}\ell e_1^\ell e_2^{10-\ell}\}$ the coefficients of $\mathcal E_\Delta(\tau,(e_1,e_2))$ are given by
$$v=\left(\begin{smallmatrix}
-17511.494570...i\\
7431.817430...\\
3204.517440...i\\
-1400.899032...\\
-619.775633...i\\
277.055319...\\
124.975219...i\\
-56.821709...\\
-26.014701...i\\
11.983426...\\
5.550045...i\end{smallmatrix}\right)\cdot 10^{-7}.$$
On the other hand we compute that $\mathcal E_\Delta(\gamma\tau)$ in the standard basis yields
$$
\left(\begin{smallmatrix}
-68.879683...+ 42.252473...i\\
 36.852496...- 27.426477...i\\ 
 -19.386784...+ 17.474359...i\\
  9.977402...- 10.932230...i\\ 
  -4.995272...+ 6.715450...i\\ 
  2.416272...- 4.051512...i\\
   -1.118039...+ 2.402667...i\\
    0.4863669...- 1.402420...i\\ 
    -0.191789...+ 0.806993...i\\ 
    0.062025...- 0.458557...i\\ 
    -0.009620...+ 0.257699...i
\end{smallmatrix}\right)=N(\gamma^{-1})v+
\left(\begin{smallmatrix}
-23814000\omega_1 \\
11895660\omega_1+12960\omega_2 \\
-5251302\omega_1  -12912\omega_2\\
1943634\omega_1    +9159\omega_2\\
-503319\omega_1-5456\omega_2\\
-14030\omega_1+2860\omega_2\\
136923\omega_1-1336\omega_2\\
-123396\omega_1+551\omega_2\\
81046\omega_1-192\omega_2\\
-45360\omega_1+48\omega_2\\
22680\omega_1
\end{smallmatrix}\right).$$
Indeed, applying $\zeta_{\Lambda_\Delta}^*(z)$ to each component of these vectors yields approximately
$$
\left(\begin{smallmatrix}
-11432504.181072... - 6.201719...\cdot 10^{-12}i\\
 6701461.733071...\\
  -6966824.048050\\
  2360012.1697371...\\ 
  6631363.825398\\ 
  10786753.122386...+ 8.786101...\cdot 10^{-33}i \\
  -8634302.260257...\\ 
  -92484.930082...\\ 
  9744076.055919... + 5.980872...\cdot 10^{-68}i \\
  -11495943.166401 - 2.586110...\cdot 10^{-9}i\\ 
  1267421.546264...
\end{smallmatrix}\right).
$$
The difference between the results for the vectors after applying $\zeta^*$ is less than $10^{-96}$, where the computations were carried out to 115 significant digits.
\subsection{A Weierstrass form for a CM form}
We consider the newform 
$$f(\tau)=\eta(3\tau)^8=q - 8q^4 + 20q^7 - 70q^{13} + 64q^{16} + 56q^{19} + O(q^{21})\in S_4(\Gamma_0(9)).$$
Note that this form has complex multiplication by the field $\QQ(\sqrt{-3})$. Note that there is a rigid Calabi–Yau threefold $X$ over $\QQ$ such that $L(X,s)=L(f,s)$ (compare \cite{wernervangeemen}).

The group $\Gamma_0(9)$ is generated by the parabolic elements $T$ and $-I_2$, as well as the matrices
$\sigma_1=\left(\begin{smallmatrix}
4 & -1 \\ 9 & -2
\end{smallmatrix}\right)$ and $ \sigma_2=\left(\begin{smallmatrix}
7 & -4 \\ 9 & -5
\end{smallmatrix}\right)$. Therefore the period lattice of $f$ is generated by the coefficients in the polynomials
\begin{align*}
p_1&=\int_{\sigma_1.\infty}^\infty f(t)(te_1+e_2)^2 dt\\
&=-0.693005\binom{2}2e_1^2 + (0.288752+ 0.033342i)\binom{2}1 e_1e_2
\\
&\quad  + (-0.115500- 0.022228i)\binom 20 e_2^2
\end{align*}
and
\begin{align*}
p_2&=\int_{\sigma_2.\infty}^\infty f(t)(te_1+e_2)^2 dt\\
&=(0.346502- 0.600160i)\binom{2}2e_1^2 + (-0.288752 + 0.433449i)\binom{2}1 e_1e_2
\\
&\quad  + (0.231001 - 0.311194i)\binom 20 e_2^2.
\end{align*}
A basis for the lattice in $\CC$ generated by these coefficients is given by $\Lambda_f=\omega_1\ZZ\oplus\omega_2\ZZ$ with
$$\omega_1=0.057750\quad\text{and}\quad \omega_2=0.011114i.$$
With this we can compute the Laurent expansion of $\widehat \zeta_{\Lambda_f}(z)$ as
\begin{multline*}
z^{-1} + 21739.040942z - 141870582.946988z^3 + 1079581634085.963275z^5+O(z^7)\\
+4894.639140\overline z.
\end{multline*}

Since the coefficient $\mathcal E_{f,0}(\tau)$ has a Fourier expansion, we can plug it into this expansion, yielding the following expression for the $0$-th component of the holomorphic part of our harmonic weak Maass form,
\begin{align*}
&\pi \left(q^{-1} + 21739.040942q + 8q^2 - 141870582.946988q^3-173912.327537q^4 +  + O(q^5)\right).
\end{align*}

We now pick some point in the upper half-plane, say $\tau=\frac{1+i\sqrt 7}{2}$. In the standard basis $\{e_1^2,2e_1e_2,e_2^2\}$ we obtain the value
$$\mathcal E_f(\tau,(e_1,e_2))=\left(\begin{smallmatrix}
5.792643+7.706733i\\
-5.792643 + 1.954292i\\
-3.908585i
\end{smallmatrix}\right)\cdot 10^{-5}.$$
On the other hand we compute 
$$\mathcal E_f(\sigma_1\tau,(e_1,e_2))=\begin{pmatrix}
-0.115037 - 0.020877i\\
 0.287651 + 0.030313i\\
  -0.690398 + 0.006789i
\end{pmatrix}=v_1+\left(\begin{smallmatrix}
-2 \\ 5 \\ -12
\end{smallmatrix}\right)\omega_1+\left(\begin{smallmatrix}
-2 \\ 3 \\ 0
\end{smallmatrix}\right)\omega_2$$
and
$$\mathcal E_f(\sigma_2\tau,(e_1,e_2))=\left(\begin{smallmatrix}
0.230596 - 0.306949i\\
 -0.288288 + 0.427988i\\
  0.345981 -  0.593136i
\end{smallmatrix}\right)=v_2+\left(\begin{smallmatrix}
4 \\ -5 \\ 6
\end{smallmatrix}\right)\omega_1+\left(\begin{smallmatrix}
-28 \\ 39  \\ -54
\end{smallmatrix}\right)\omega_2.$$
Plugging either one of the vectors above into $\widehat\zeta_{\Lambda_f}(z)$ yields 
$$\left(\begin{smallmatrix}
235.526014 - 626.338523i\\
 -129.035666+ 208.728000i\\
  165.268142 + 59.208840i
  \end{smallmatrix}\right) \quad\text{resp.}\quad 
  \left(\begin{smallmatrix}
  -33.090396 - 108.245540i\\
   32.306903 + 7.539680i\\
    -44.064442 + 121.275441i
  \end{smallmatrix}\right).$$

\bibliographystyle{alpha}
\bibliography{bib.bib}

\end{document}